\documentclass[a4paper,11pt]{article}
\usepackage[utf8]{inputenc}
\usepackage[margin=30mm]{geometry}

\usepackage{amsmath,amsthm}
\usepackage{tikz-cd}

\usepackage{enumerate}
\usepackage{mathtools}
\usepackage{stmaryrd}
\usepackage[table,dvipsnames]{xcolor}
\usepackage[colorlinks=true, linkcolor=MidnightBlue, urlcolor=Blue, citecolor=MidnightBlue]{hyperref}
\usepackage[charter]{mathdesign}

\newcommand{\Csr}{\mathrm{C}_{\,\mathrm{r}}^*}
\newcommand{\Cs}{\mathrm{C}^*}
\newcommand{\bounded}{\operatorname{\mathfrak{B}}}
\newcommand{\compact}{\operatorname{\mathfrak{K}}}
\newcommand{\R}{\mathbb{R}}

\newcommand{\Ker}{\operatorname{Ker}}
\newcommand{\scal}[2]{\left\langle#1\,,\,#2\right\rangle}

\newcommand{\Prim}{\operatorname{Prim}}
\newcommand{\dom}{\operatorname{dom}}
\newcommand{\proj}{\operatorname{proj}}
\newcommand{\Ind}{\operatorname{Ind}}
\newcommand{\seq}{\mathrm{seq}}
\newcommand{\Iminfch}{\operatorname{Im\,inf\,ch}}
\newcommand{\adomstar}{\mathfrak{a}_{\dom}^\ast}
\newcommand{\aiplus}{\mathfrak{a}_{I, +}^\ast}
\newcommand{\hatMItempiric}{\widehat{M}_{I, \mathrm{tempiric}}}
\newcommand{\hatGtemp}{\widehat{G}_{\mathrm{tempered}}}
\newcommand{\Gtempiric}{\widehat{G}_{\mathrm{tempiric}}}
\newcommand{\hatGzero}{\widehat{G}_0}
\newcommand{\hatKI}{\widehat{K}_{I}}
\newcommand{\Khat}{\widehat{K}}
\newcommand{\LKT}{\operatorname{LKT}}
\newcommand{\mult}{\operatorname{mult}}
\newcommand{\amin}{\mathfrak{a}_{\min}}
\newcommand{\Pmin}{P_{\min}}
\newcommand{\Mmin}{M_{\min}}
\newcommand{\Amin}{A_{\min}}
\newcommand{\Nmin}{N_{\min}}
\newcommand{\nmin}{\mathfrak{n}_{\min}}
\newcommand{\doublequo}{\!/\!\!/}

\newcommand{\UI}{{U}_{I}}

\newcommand{\Knorm}[1]{\lVert{#1}\rVert_{\widehat{K}}}

\newcommand{\Ups}{\Upsilon}

\newcommand{\N}{\mathbb{N}}

\numberwithin{equation}{section}

\theoremstyle{plain}
\newtheorem{theorem}[equation]{Theorem}
\newtheorem{lemma}[equation]{Lemma}
\newtheorem{corollary}[equation]{Corollary}
\newtheorem{proposition}[equation]{Proposition}
\newtheorem*{theorem*}{Theorem}
\newtheorem*{proposition*}{Proposition}
\newtheorem*{corollary*}{Corollary}
\newtheorem*{lemma*}{Lemma}
\newtheorem*{claim*}{Claim}
\newtheorem*{theorema*}{Theorem A}
\newtheorem*{theoremb*}{Theorem B}

\newenvironment{theorembis}[1]
  {\renewcommand{\theequation}{\ref*{#1}$'$}%
   \addtocounter{equation}{-1}%
   \begin{theorem}}
  {\end{theorem}}

\theoremstyle{definition}
\newtheorem{definition}[equation]{Definition}
\newtheorem{example}[equation]{Example}

\newtheorem{notation}[equation]{Notation}
\newtheorem{remark}[equation]{Remark}

\title{The Mackey bijection as a stratified equivalence}
\author{Alexandre Afgoustidis\footnote{CNRS \& Institut Élie Cartan de Lorraine, Nancy \& Metz, France; \texttt{alexandre.afgoustidis@math.cnrs.fr}}\ \ \& Pierre Clare\footnote{College of William \& Mary, Williamsburg, VA, USA; \texttt{prclare@wm.edu}}}
\date{\empty}
\begin{document}

\maketitle

\begin{abstract}
This paper is about the Mackey analogy between the tempered representation theory of a real reductive group and that of its Cartan motion group. We consider the embedding of reduced C*-algebras constructed recently in connection with the Mackey bijection, and study its behavior on certain natural stratifications of the tempered duals. We formulate our result using a notion of stratified equivalence inspired by the study of the smooth dual of $p$-adic groups via the structure of Hecke algebras, in particular by the work of Aubert, Baum, Plymen and Solleveld. We derive related new topological properties of the Mackey bijection. We also analyze the behavior of the Mackey embedding on a stratification of reduced C*-algebras attached to a partition of the tempered dual into particularly elementary pieces, introduced in recent work of Bradd, Higson and Yuncken.
\end{abstract}

\section*{Introduction}

Let $G$ be a real reductive group.  The Mackey analogy connects the (tempered) representation theory of~$G$ with the (unitary) representation theory of its Cartan motion group. 
Since its inception by George Mackey in the 1970s as a tentative statement on the parametrization of representations, it has been studied from various points of view, including the theory of operator algebras and its interplay with canonical topologies on tempered duals \cite{Mackey75, Higson_08, Mackey_bijection}.

Among the byproducts of this evolution 
is some insight into the geometry of the tempered dual $\hatGtemp$:  the space $\hatGtemp$ (which is in general not Hausdorff in the Fell topology) can be assembled from a countable collection of contractible convex cones in finite-dimensional vector spaces. There are at least two distinct, but related, ways of doing this: in terms of lowest $K$-types, as in \cite{Mackey_CK_isom}, or in terms of the imaginary part of the infinitesimal character, as in \cite{BHY}. 

For reductive $p$-adic groups, similar phenomena occur in the work of Aubert--Baum--Plymen--Solleveld on the geometry of each Bernstein component of the smooth dual \cite{ABPS_Takagi, ABPS_Cstar}, showing that each individual component is usually assembled from elementary pieces of a slightly different kind. In this $p$-adic setting, a language gradually emerged to describe the phenomena, using a notion of ``{stratified equivalence}'' \cite{ABPS_Morita} to relate the Hecke algebras attached to Bernstein blocks to much simpler algebras. 

In this paper we use the Mackey embedding recently constructed in \cite{CHR} to recast the Mackey analogy in the language that was conceived for $p$-adic groups. We also prove various properties of the Mackey bijection \cite{Mackey_bijection} and of the Mackey embedding \cite{CHR}, in relation with various stratifications of the reduced $\Cs$-algebra of~$G$.

Before we describe our results, let us provide some background on the Mackey analogy for real groups, and on the $p$-adic phenomena that underlie the notion of stratified equivalence.

\paragraph*{The Mackey bijection and the Mackey embedding.} Let $K$ be a maximal compact subgroup of~$G$. The Cartan motion group $G_0$ attached to $G$ and~$K$ is the semidirect product $K \ltimes (\mathfrak{g}/\mathfrak{k})$, built from the adjoint action of $K$ on the vector space quotient $\mathfrak{g}/\mathfrak{k}$ of the Lie algebras of~$G$ and~$K$.  The  group $G_0$ has  the same dimension as~$G$, but a large abelian normal subgroup; therefore its unitary dual $\hatGzero$ is easy to describe using Mackey's ``little group'' method. The tempered dual $\hatGtemp$ is more subtle. But a few years ago, building on suggestions of Mackey \cite{Mackey75} and Higson \cite{Higson_08}, one of us  described in \cite{Mackey_bijection} a natural bijection  
\[ \hatGzero \longleftrightarrow \hatGtemp\] using Vogan's notion of lowest $K$-types. This ``Mackey bijection'' is not a homeomorphism in the Fell topology (unless $G$ and $G_0$ are isomorphic), but it is a \emph{piecewise homeomorphism}: for each finite collection~$E$ of $K$-types, it induces a homeomorphism between the subsets consisting of the representations whose set of lowest $K$-types is exactly~$E$ \cite{Mackey_CK_isom, Mackey_continuity}. 

The Mackey bijection has long been known to have connections the Connes--Kasparov isomorphism in operator algebras \cite{BCH, Higson_08, Mackey_CK_isom}: it was Alain Connes who pointed out in the 1980s that the Connes--Kasparov isomorphism for real reductive groups can be interpreted as a $K$-theoretic reflection of Mackey's ideas.  In recent work of another of us with Higson and Rom\'an  \cite{CHR}, building on ideas introduced in Rom\'an's thesis, the  bijection was promoted to a $\Cs$-algebraic embedding 
\[ \alpha\colon \Cs(G_0) \longrightarrow \Csr(G)\] of the maximal $\Cs$-algebra of $G_0$ into the reduced $\Cs$-algebra of~$G$. The Mackey bijection is then recovered by taking an irreducible tempered representation of~$G$, viewed as a representation~$\pi$ of~$\Csr(G)$; composing it with the embedding $\alpha$; and, viewing the composition $\pi\circ\alpha$ as a representation of~$G_0$, extracting the irreducible component that contains the lowest $K$-types. One important property of the embedding~$\alpha$ is that it induces the Connes--Kasparov isomorphism in K-theory, giving operator-algebraic substance to the idea that $\hatGzero$~and $\hatGtemp$ are similar. In this spirit, it is also known \cite{Mackey_CK_isom} that $\Cs(G_0)$ and $ \Csr(G)$ are assembled from very simple common pieces by Morita equivalences and direct limits; as a result, the homeomorphic pieces of $\hatGzero$ and $\hatGtemp$ discussed above are assembled in such a way that the two duals share algebraic-topological invariants.

\paragraph*{Bernstein components for $p$-adic groups; stratified equivalence.} In the past 20 years, but quite independently of the Mackey analogy, conceptually similar phenomena took on an increasingly important role in certain aspects of the representation theory  of reductive $p$-adic groups. These appear when one focuses on the geometry of a \emph{single connected component} $\mathfrak{s}$ of the smooth dual, the space of equivalence classes of irreducible smooth representations. One can attach to the  component $\mathfrak{s}$ a Hecke algebra $\mathscr{H}_{\mathfrak{s}}$, a complex algebra whose primitive ideal space identifies with~$\mathfrak{s}$. Under some tameness assumptions and in many cases, including the  emblematic case where $\mathfrak{s}$ is the Iwahori component, the algebra $\mathscr{H}_{\mathfrak{s}}$ is Morita-equivalent to an affine Hecke algebra $\widetilde{\mathscr{H}}_{\mathfrak{s}}$, whose structure therefore encapsulates much of the representation theory.

It gradually emerged from the work of Lusztig \cite{Lusztig} and of Baum--Nistor \cite{BaumNistor} for the Iwahori component, and of Aubert--Baum--Plymen--Solleveld in general (see the surveys \cite{ABPS_Takagi, ABPS_Cstar}), that $\widetilde{\mathscr{H}}_{\mathfrak{s}}$ has a remarkable relation with a simpler algebra $\widetilde{\mathscr{H}}_0$ that is reminiscent of the $\Cs$-algebra~$\Cs(G_0)$ in the real case. This  $\widetilde{\mathscr{H}}_0$ is the group algebra $\mathbb{C}[W \ltimes X]$, where $X$ is the character lattice of a compact torus and $W$ is a finite group acting on $X$. The algebras $\widetilde{\mathscr{H}}_{\mathfrak{s}}$ and $\widetilde{\mathscr{H}}_0$ share the same underlying vector space, but the algebra structure of~$\widetilde{\mathscr{H}}_{\mathfrak{s}}$ is a deformation of that of $\widetilde{\mathscr{H}}_0$, and much more complicated.  Nevertheless, it turns out that the primitive ideal spaces of $\mathscr{H}_\mathfrak{s}$ and $\widetilde{\mathscr{H}}_0$ are in natural bijection, and piecewise homeomorphic. 

To discuss such phenomena, Baum--Nistor and Aubert--Baum--Plymen--Solleveld introduce a notion of ``{stratified equivalence}''. At its heart is a notion that we shall call ``stratified equivalence morphism'' below (see Section~\ref{sec-strat-defs} for details): a stratified equivalence morphism between two algebras $A$ and~$B$ is an algebra homomorphism which respects certain filtrations $(A_i)$, $(B_i)$ of $A$ and~$B$, and which on subquotients has the  special property that the preimage of any primitive ideal in a subquotient $B_{i}/B_{i-1}$ is contained in a unique primitive ideal of $A_i/A_{i-1}$. Such a morphism induces a natural bijection between the primitive ideal spaces of~$A$ and~$B$. If $A$ and~$B$ are finitely generated and the filtrations are finite, Baum and Nistor show that this bijection is a piecewise homeomorphism; more precisely, that it induces a homeomorphism between the spectra of $B_{i}/B_{i-1}$  $A_i/A_{i-1}$ for all $i$, but not a global homeomorphism in general.  

Aubert, Baum, Plymen and Solleveld have highlighted the role of this notion in the $p$-adic theory. For instance, if $\mathfrak{s}$ is the Iwahori component of a reductive $p$-adic group, then each of $\mathscr{H}_\mathfrak{s}$ and $\widetilde{\mathscr{H}}_0$ can be related, via a stratified equivalence morphism, to a version of Lusztig's asymptotic Hecke algebra. This plays an important role for the discussion of the geometric structure of $\mathfrak{s}$ in the work of Aubert, Baum, Plymen and Solleveld; and also in the computation of the $K$-theory of $\mathscr{H}_\mathfrak{s}$, which  sheds light on the analytic side of the Baum--Connes conjecture. See \cite{ABPS_Cstar, SolleveldTrimestre}.

\paragraph*{Our results.} The striking parallels between the real and $p$-adic phenomena described above have been known to the experts for a few years. However, by providing a concrete map at the level of group algebras, the construction of the Mackey embedding $\alpha\colon \Cs(G_0)\longrightarrow \Csr(G)$ in \cite{CHR} now makes it possible to study these parallels more formally. In this paper, we shall prove a first result in this direction, and see that it reveals new properties of the Mackey bijection.

In Section~\ref{sec-strat-eq} we consider two pairs  of filtrations for $\Cs(G_0)$ and $\Csr(G)$, which we shall denote by $(I_n^0)_{n \in \N}$, $\left(I_n\right)_{n \in \N}$ and $(\Ups^0_k)_{k \in \N}$, $(\Ups_{k})_{k \in \N}$ respectively. Both are defined in terms of lowest $K$-types. The filtrations  $(I_n^0)_{n \in \N}$, $\left(I_n\right)_{n \in \N}$ of $\Cs(G_0)$ and $\Csr(G)$ are those which appeared in \cite{Mackey_CK_isom} in connection with the Connes--Kasparov isomorphism. The filtrations $(\Ups^0_k)_{k \in \N}$ and $(\Ups_{k})_{k \in \N}$ are coarser and only keep track of the ``height'' of lowest $K$-types. We prove the following result (see Theorem~\ref{th-strat-LKT}).

\begin{theorema*} \label{th-a} The Mackey embedding  $\alpha\colon \Cs(G_0)\longrightarrow \Csr(G)$ is a stratified equivalence morphism with respect to the filtrations $\left(I_n^0\right)_{n \in \N}$ and $\left(I_n\right)_{n \in \N}$ of $\Cs(G_0)$ and $\Csr(G)$. It is also a stratified equivalence morphism  with respect to the filtrations $\left(\Ups^0_k\right)_{k \in \N}$ and $(\Ups_{k})_{k \in \N}$.
\end{theorema*}

The bijection between $\hatGtemp$ and $\hatGzero$ induced by the stratified equivalence property is, in both cases, the Mackey bijection. 
Once this is proved, the work of Baum and Nistor \cite{BaumNistor} strongly suggests that the bijection is a piecewise homeomorphism, and induces a homeomorphism between pieces of $\hatGzero$ and $\hatGtemp$ corresponding to the spectra of our filtrations' subquotients. In the case of the filtrations  $(I_n^0)_{n \in \N}$, $\left(I_n\right)_{n \in \N}$ this was already known \cite{Mackey_CK_isom}, and amounts to the above statement on the behavior of the Mackey bijection when a fixed set of lowest $K$-types is prescribed. But for the coarser filtration, this reveals a stronger property of the Mackey bijection which is, to our knowledge, new. It amounts to the following statement (see Corollary~\ref{cor-homeo-composantes}).

\begin{theoremb*} \label{th-b} The Mackey bijection maps  each connected component of $\hatGtemp$ homeomorphically onto its image in $\hatGzero$. 
\end{theoremb*}

Thus the unitary dual $\hatGzero$ is obtained by gluing together the components of~$\hatGtemp$ in a rather peculiar manner, that preserves algebraic-topological invariants. Although this result was suggested to us by stratified equivalence, in Section~\ref{sec-topology} we shall give a direct representation-theoretic proof. 

\medskip

Up to this point our discussion has been focused on two partitions of $\hatGzero$ and $\hatGtemp$ according to lowest $K$-types, and on the corresponding stratifications of $\Cs(G_0)$ and $\Csr(G)$. While this is very useful to understand the Mackey bijection, recent work of Bradd, Higson and Yuncken \cite{BHY} has shown that much further insight on the bijection can be gained by considering another way to partition  $\hatGzero$ and $\hatGtemp$, according to the ``singularity'' of the imaginary part of the infinitesimal character. This makes the construction of \cite{Mackey_bijection} quite transparent, and in our discussion of the Mackey bijection in Section~\ref{sec-prelim} we will take up notation and ideas from \cite{BHY}. 

The partition of $\hatGzero$ and $\hatGtemp$ according to the imaginary part of the infinitesimal character also has its own operator-algebraic counterpart. On their way to proving the main results of \cite{BHY}, the authors consider corresponding filtrations of $\Cs(G_0)$ and~$\Csr(G)$. For a given filtration piece,  the individual subquotients $\Cs(G_0; I)$ and $\Csr(G; I)$ both have spectrum the set of representations whose imaginary part of the infinitesimal character lies on a specific wall in the dominant chamber of all possible imaginary parts. Bradd, Higson and Yuncken point out that these subquotients have a very simple structure up to $\Cs$-algebra isomorphism, and that $\Cs(G_0; I)$ and~$\Csr(G; I)$ are Morita-equivalent. In this way the infinitesimal character leads to a partition of $\hatGzero$ and~$\hatGtemp$ into yet another set of contractible homeomorphic pieces, assembled differently in both duals. 

The proofs of the main results in \cite{BHY} feature their own versions of the Mackey embedding; these are much simpler than the ``global'' Mackey embedding $\alpha\colon \Cs(G_0) \longrightarrow \Csr(G)$ of \cite{CHR}, but they are defined only at the level of the subquotients $\Cs(G_0; I)$ and $\Csr(G; I)$ using the very simple structure of these subquotients. Since a large portion of this paper is about the behavior on the Mackey embedding~$\alpha$ on various filtrations of the $\Cs$-algebras of $G_0$ and~$G$, we have taken this opportunity to clarify the relationship between the ``global'' embedding constructed in \cite{CHR} and the ``local'' versions that appear in \cite{BHY}. This is the theme of  Section~\ref{sec-localizations}.

\paragraph*{Organization of the paper.}  Section~\ref{sec-prelim} collects the definitions and results on the tempered dual, and on the Mackey bijection and embedding, that will be used throughout the paper. As mentioned above, the influence of the recent work of Bradd, Higson and Yuncken will be apparent in our choices of notation and presentation. 

After that, Sections \ref{sec-strat-eq}, \ref{sec-topology} and \ref{sec-localizations} can be read more or less independently of one another. Section~\ref{sec-strat-eq} is about the stratified equivalence property of the Mackey embedding; there we state the appropriate definitions (indicating the slight deviations from those of Baum--Nistor and Aubert--Baum--Plymen--Solleveld), define the filtrations $(\Ups_k)_{k \in \N}$ and $\left(I_n\right)_{n \in \N}$, and prove Theorem~\hyperref[th-a]{A}. Section~\ref{sec-topology} contains the proof of Theorem~\hyperref[th-b]{B}. While the result was suggested to us by the inspection of stratified equivalence for the filtrations $(\Ups_k)_{k \in \N}$, we give a purely representation-theoretic proof that does not refer to the notions of Section~\ref{sec-strat-eq}. Finally, in Section~\ref{sec-localizations} we discuss the infinitesimal character filtration of Bradd, Higson and Yuncken before proving a comparison theorem between the Mackey embedding $\alpha$ and the ``localized'' versions on the subquotients of the infinitesimal character filtration, and collecting various remarks on the (non-)uniqueness properties of the embedding $\alpha$.

\paragraph*{Acknowledgements.} This paper grew out of a discussion that took place on the 28th of February, 2025, during a workshop on tempered representations and $K$-theory at Institut Henri Poincaré. We thank all the colleagues who took part in that discussion, in particular Anne-Marie Aubert, Roger Plymen and Maarten Solleveld. We are also indebted to Jacob Bradd, Nigel Higson and Robert Yuncken for several enlightening conversations over the past year; Theorem~\hyperref[th-b]{B} was brought on by questions of Bradd and Yuncken. 

 We extend our gratitude to the Institut Henri Poincaré  (UAR 839 CNRS-Sorbonne Université) and its Centre Émile Borel, as well as the project OpART of the Agence Nationale de la Recherche (ANR-23-CE40-0016), for their help and support in relation with the thematic trimester ``Representation theory and noncommutative geometry'', and during the preparation of this paper. P. C. is particularly grateful to the Institut Élie Cartan de Lorraine and the Fondation Sciences Mathématiques de Paris for supporting him through visiting professorships.

\section{The tempered dual and the Mackey bijection}\label{sec-prelim}

\subsection{\texorpdfstring{Parameters for the tempered dual of~$G$}{Parameters for the tempered dual of G}}

Throughout the paper, by a \emph{real reductive group} we shall mean the group of real points of a connected reductive algebraic group defined over $\mathbb{R}$; and we shall fix throughout a real reductive group~$G$, together with a Cartan involution $\theta$ of~$G$. We denote by $K$ the set of fixed-points $G^\theta$. It is a  maximal compact subgroup of $G$ and we write $\mathfrak{g}=\mathfrak{k}\oplus\mathfrak{p}$ for the Cartan decomposition of the Lie algebra $\mathfrak{g}=\operatorname{Lie}(G)$, where $\mathfrak{k}$ is the Lie algebra of~$K$. We assume throughout that  the Lie algebra~$\mathfrak{g}$ is equipped with  a symmmetric, non-degenerate, $G$-invariant bilinear form that is positive-definite on $\mathfrak{p}$ and negative-definite on $\mathfrak{k}$. This gives positive-definite inner products on each of $\mathfrak{k}$ and $\mathfrak{p}$; we shall denote these inner products and their restrictions to various subspaces by $\left\langle\cdot\,,\,\cdot\right\rangle$, and the corresponding norms by $\lVert{\cdot}\rVert$, without any further reference. In particular, if $\mathfrak{a}_1 \subset \mathfrak{a}_2$ are nested subspaces of~$\mathfrak{p}$,  any linear functional $\nu$ on $\mathfrak{a}_1$ extends uniquely to a linear functional on~$\mathfrak{a}_2$ which is zero on the orthogonal complement of~$\mathfrak{a}_1$ in~$\mathfrak{a}_2$.  We shall call this ``the'' extension of~$\nu$ to~$\mathfrak{a}_2$ without further comment.

We also fix a maximal abelian subspace $\amin\subset \mathfrak{p}$ and a system of positive restricted roots \[\Delta^+(\mathfrak{g},\amin)\subset\Delta(\mathfrak{g},\amin),\] which will remain fixed throughout the paper.
They determine an Iwasawa decomposition \[G=K\Amin\Nmin\]
where $\Amin=\exp(\amin)$ and $\Nmin=\exp(\mathfrak{\nmin})$, with $\nmin$ defined to be the direct sum of root spaces $\bigoplus\limits_{\alpha\in\Delta^+(\mathfrak{g},\amin)}\mathfrak{g}_\alpha$. Finally, we denote by $\Mmin$ the centralizer of $\amin$ in $K$; then 
 \[\Pmin = \Mmin \Amin\Nmin\] is a minimal parabolic subgroup of~$G$.   Having fixed these elements of notation, let us now introduce  parameters that will allow us to describe the tempered representation theory of~$G$.

\subsubsection{Imaginary part of the infinitesimal character; tempiric representations}\label{sec-iminfch}

Let $S\subset\Delta^+(\mathfrak{g},\amin)$ denote the set of simple restricted roots. As in \cite{BHY}, subsets of $S$ will play a central role in the parametrization of the tempered dual of $G$ that we use below.

A \emph{parabolic subgroup} of $G$ is a closed subgroup containing a conjugate of $\Pmin$. Such a group~$P$ is called \emph{standard} if $\Pmin\subset P$, and it is well known that standard parabolic subgroups of~$G$ are in natural one-to-one correspondence with subsets of $S$, as follows. Any subset $I\subset S$ determines a unique standard parabolic subgroup $P_I$ with Langlands decomposition \[P_I=M_IA_IN_I\] where  $A_I = \exp(\mathfrak{a}_I)$ for $\mathfrak{a}_I=\bigcap_{\alpha\in I}\Ker \alpha$, where the Levi component $L_I=M_IA_I$ of $P_I$ is the centralizer in $G$ of $\mathfrak{a}_I$, where $M_I$ is the subgroup of $L_I$ generated by all compact subgroups of $L_I$, and where the unipotent radical~$N_I$ is $\exp(\mathfrak{n}_I)$ with $\mathfrak{n}_I=\bigoplus_{\alpha\in S\setminus I}\mathfrak{g}_\alpha$. See for instance \cite[Chapter VII, \S~7]{Knapp_Beyond} for further details.

We set \[\adomstar=\left\lbrace\, \nu\in\amin^*\::\:\scal{\nu}{\alpha}\geq0\,,\,\forall \alpha\in S\, \right\rbrace;\]
this is the dominant Weyl chamber in $\amin^*$ associated with the given choice of positive roots.

\begin{remark}
For any parabolic subgroup $P$ of $G$ with Langlands decomposition $P=MAN$ the dual of the Lie algebra $\mathfrak{a}$ of $A$ can be viewed as a subspace of $\amin^*$, by extending a linear functional on $\mathfrak{a}$ to a functional on $\amin^*$ that vanishes on the orthogonal complement $\amin\cap\mathfrak{m}$ of $\mathfrak{a}$ in $\amin$. This gives a map \begin{equation}\label{eq-inclusion-amin*}
\mathfrak{a}^*\longrightarrow\amin^*.
\end{equation}
When $\mathfrak{a}=\mathfrak{a}_I$ for some $I\subset S$, the image is $\mathfrak{a}_I^*=\left\lbrace \gamma\in\amin^*\::\:\scal{\alpha}{\gamma}=0 \,,\,\forall \alpha\in I
\right\rbrace$.
\end{remark}

Still following \cite{BHY} (see Section~2.3 there), for $I\subset S$, we  further define 
\[\aiplus=\left\lbrace\, \nu\in\adomstar\::\:\scal{\nu}{\alpha}=0\,,\,\forall \alpha\in I\quad\text{and}\quad\scal{\nu}{\beta}\neq0\,,\,\forall \beta\in S\setminus I\, \right\rbrace.\]
We shall need the following properties of the sets $\aiplus$. 
\begin{enumerate}[(a)]
\item For fixed $I \subset S$, the set $\aiplus$ is contained in $\mathfrak{a}_I^*$, and we have \[\aiplus=\adomstar\cap \mathfrak{a}_I^*\setminus\bigcup_{I\subsetneq J} \mathfrak{a}_J^*.\]
\item The set $\mathfrak{a}^*_\varnothing$ is the interior of $\adomstar$, whereas $\mathfrak{a}^*_S$ identifies with the intersection of $\amin$ with the center of $\mathfrak{g}$.
\item The sets $\aiplus$ are locally closed in $\adomstar$, and give a partition of $\adomstar$:
\begin{equation} \label{eq-partition-chambre}\adomstar=\bigsqcup_{I\subset S}\aiplus.\end{equation}
\end{enumerate}

Let us now turn to representations. If $\pi$ is an irreducible unitary representation of $G$, we shall denote the \emph{imaginary part of its infinitesimal character} by \[\Iminfch(\pi)\in\amin^*/W(\mathfrak{g},\amin).\] See \cite[Section~2.2]{BHY} for details regarding this notion. Although the basic ideas probably date back to Harish-Chandra, the present formulation draws on work of Vogan \cite{Vogan81, Vogan00}.

For our purposes it will be convenient to view the imaginary part of the infinitesimal character as an element of the chamber $\adomstar$, using the fact that the latter is a fundamental domain for the action of~$W$ on $\amin^\ast$: the inclusion of $\adomstar$ into $\amin^\ast$ induces a bijection

\begin{equation}\label{eq-chambre-dom} \adomstar \overset{\sim}{\longleftrightarrow}  \amin^\ast/W. \end{equation}

Therefore it makes sense to use the decomposition~\eqref{eq-partition-chambre} to partition the unitary dual of~$G$ according to which subset $\aiplus$ contains the imaginary part of the infinitesimal character. This is one of the main ingredients of \cite{BHY}. 

The following extreme case is crucial to discuss the Mackey bijection.

\begin{definition}\label{def-rep-tempirique}
An irreducible unitary representation $\pi$ of a real reductive group is said to have \emph{real infinitesimal character} if it satisfies $\Iminfch(\pi) = 0$.  It is said to be \emph{tempiric} if it is tempered, irreducible, and has real infinitesimal character.
We denote by $\Gtempiric$ the set of equivalence classes of tempiric representations of $G$. 

\end{definition}

In Section~\ref{sec-parametrisation-bhy} below we shall outline a description of the tempered dual $\hatGtemp$ in terms of tempiric representations of Levi subgroups.  This relies on the  following fact, again drawing on work of Vogan, which provides a classification of tempered irreducible representations in terms of the imaginary part of their infinitesimal character. In the statement we shall freely use notation related to parabolic induction; some of the basics will be recalled in Section~\ref{sec-parametrisation-bhy}. 

\begin{theorem}[See {\cite[Theorem 2.4.2]{BHY}}]\label{th-strat-par-ImInfChar}
For fixed $I\subset S$ and $\nu\in\aiplus$, the map \[\sigma\longmapsto\Ind_{P_I}^G\sigma\otimes e^{i\nu}\otimes 1\] determines a bijection between the set $\hatMItempiric$ of unitary equivalence classes of tempiric representations $\sigma$ of~$M_I$, on the one hand, and the set of unitary equivalence classes of irreducible tempered representations $\pi$ of~$G$ for which $\Iminfch(\pi)=\nu$, on the other hand.
\end{theorem}

See \cite[Appendix 2]{BHY} for a sketch of proof of this result, due to Vogan in a greater generality, and for further comments.

\subsubsection{\texorpdfstring{Lowest $K$-types}{Lowest K-types}}\label{sec-basics_LKT}

Alongside the imaginary part of the infinitesimal character, another way to sort unitary representations of $G$ is to study their $K$-types. Let us summarize here the elements of Vogan's theory of lowest $K$-types that we shall use later. See \cite{Vogan85}, and e.g. \cite[Section 6.1]{BHY} or \cite[Section 4.2]{CHR}, for additional details and references.

\begin{definition} 
\label{def-norme-sur-K^}
Let $K$ be a compact Lie group and let $T$ be a  maximal torus in the identity component of $K$.  Fix a system of positive roots for $(\mathfrak{k}, \mathfrak{t})$, and let $\rho_K$ be the half-sum of positive roots.  Fix a $K$-invariant inner product on $\mathfrak{k}$.  If $\lambda$ is an irreducible representation of $K$, then we define its \emph{height} by
\[\Knorm{\lambda} = \| \mu {+} 2 \rho_K\|,\]
where $\mu$ is any highest weight of $\lambda$. 
\end{definition}

\begin{definition}
Let $\pi$ be an admissible representation of $G$. A \emph{lowest K-type} of $\pi$ is any $K$-type whose norm is minimal among the norms of all $K$-types of $\pi$.
\end{definition}

Norm-bounded subsets of $\Khat$ are finite. It follows that every admissible representation has at least one lowest $K$-type, and at most finitely many. While the height function $\Knorm{\cdot}$ in Definition~\ref{def-norme-sur-K^} depends on the chosen inner product and on the chosen set of positive roots, the 
lowest $K$-types of~$\pi$ depend only on $\pi$, not on the other choices.

\begin{notation}
Given an admissible representation $\pi$ of $G$, we denote its set of lowest $K$-types by \[\LKT_G(\pi)=\left\lbrace\, \lambda\in\Khat\::\:\text{$\lambda$ is a lowest $K$-type of $\pi$}\, \right\rbrace\] or simply $\LKT(\pi)$ when no ambiguity is likely to arise from doing so. Furthermore, since all elements in $\LKT(\pi)$ have the same height, we shall write \[\Knorm{\LKT(\pi)} \]
for the common height $\Knorm{\lambda}$ of all elements $\lambda\in\LKT(\pi)$.
\end{notation}

\begin{theorem}[Vogan \cite{Vogan81}; see \cite{Vogan07}, Theorem 1.2]\label{th-vogan-tempiric}~
\begin{enumerate}[(a)]
\item If $\pi$ is a tempiric representation of~$G$, then $\LKT(\pi)$ contains a single element of $\Khat$.
\item The assignement $\pi \longmapsto \LKT_G(\pi)$ induces a bijection 
\[\Gtempiric\overset{\simeq}{\longrightarrow}\Khat.\]
\end{enumerate}
\end{theorem}

\subsubsection{Parametrization of the tempered dual}\label{sec-parametrisation-bhy}

If $P$ is a parabolic subgroup of $G$ with Langlands decomposition $P=MAN$, and if we are given a tempered representation $\sigma$ of $M_I$ and an element $\nu$ of $\mathfrak{a}^+$, then the pair $(\sigma,\nu)$ determines a tempered representation $\sigma\otimes e^{i\nu}$ of the Levi component $L=MA$, which can be extended as $\sigma\otimes e^{i\nu}\otimes 1_N$ to $P$ and unitarily induced to $G$. The representation $\Ind_P^G\sigma\otimes e^{i\nu}\otimes 1_N$ obtained in this way is said to be \emph{parabolically induced} from $(\sigma,\nu)$.

These representations can be realized in various Hilbert spaces -- see for instance \cite{vandenBan_Induced_reps}. We shall use the ``compact picture'', which has the advantage that the carrier Hilbert space is independent of the parameter $\nu$. It is obtained as the set of $K \cap L$-equivariant vectors in the space of square-integrable, $H_\sigma$-valued functions on $K$. The action on~$G$ on this space is defined by means of the Iwasawa decomposition, but we shall not need the precise formula. For brevity's sake, as in \cite[Section 3]{CHR} we shall denote this space by \begin{equation}\label{eq-def-Htau}
\Ind H_\sigma =  \operatorname{L^2} (K, H_\pi)^{K \cap L}.
\end{equation}

In order to describe the tempered dual of $G$ in a language suited to our purposes, it will be useful to consider parabolically induced representations labeled by subsets of roots, as in Section \ref{sec-iminfch}.

\begin{notation}
If $I \subset S$ is a subset of simple roots, $\tau$ is a tempiric representation of $M_I$, and $\nu \in \aiplus$, we shall write \[\pi_{I, \tau, \nu}=\Ind_{P_I}^G\tau\otimes e^{i\nu}\otimes 1_{N_I}.\]
\end{notation}

\begin{remark}
If $I$ and $\tau$ are fixed then the representations $\pi_{I, \tau, \nu}$, as $\nu$ ranges over $\aiplus$, all have the same restriction to~$K$.  In particular, the set $\LKT(\pi_{I, \tau, \nu})$ does not depend on $\nu$. 
\end{remark}

The following statement amounts to a classification of the irreducible tempered representations of~$G$:
\begin{theorem}\label{th-bhy-classification}
\begin{enumerate}[(a)]
\item For every triple $(I, \tau, \nu)$ where  $I \subset S$, where $\tau$ is a tempiric representation of~$M_I$, and where $\nu \in \aiplus$, the representation $\pi_{I, \tau, \nu}$ of~$G$ is irreducible. 
\item Every irreducible tempered representation of~$G$ is equivalent with one of the representations~$\pi_{I, \tau, \nu}$.
\item Given triples $(I, \tau, \nu)$ and $(I', \tau', \nu')$ where  $I, I'\subset S$, where $\tau$, $\tau'$ are tempiric representations of $M_I$ and $M_{I'}$ respectively, and  where $\nu \in \aiplus$, $\nu' \in \mathfrak{a}^\ast_{I, +}$, we have  $\pi_{I, \tau, \nu} \simeq \pi_{I', \tau', \nu'}$ if and only if $I=I'$, $\tau \simeq \tau'$ and $\nu=\nu'$.
\end{enumerate}
\end{theorem}

The three assertions in the theorem encapsulate everything that was used in \cite{Mackey_bijection} to prove that the construction there provides a bijection between $\hatGtemp$ and the unitary dual of~$G_0$ ; see the discussion in \cite[proof of Theorem~3.5]{Mackey_bijection} for reformulations of the relevant parts of the literature concerning each assertion.  
The version of the classification formulated in Theorem~\ref{th-bhy-classification} is implicitly used throughout \cite{BHY}. 

\subsection{\texorpdfstring{The unitary dual of~$G_0$ and the Mackey bijection}{The unitary dual of G0 and the Mackey bijection}}\label{sec-params-G0}

\subsubsection{Unitary dual of the Cartan motion group}

\paragraph{Construction of unitary representations of $G_0$.} 
In the rest of the paper, we define the Cartan motion group $G_0$ for~$G$ (and~$K$) to be the semidirect product $K \ltimes \mathfrak{p}$ attached to the adjoint action of~$K$ on~$\mathfrak{p}$. Let us recall Mackey's recipe for constructing unitary representations of~$G_0$, and his classification of irreducible unitary representations. 

The main tool is the coadjoint action of~$K$ on the dual vector space $\mathfrak{p}^\ast$. Suppose given a pair $(\chi, \mu)$ where $\chi$ is an element of $\mathfrak{p}^\ast$ and $\mu$ is a unitary representation of the stabilizer~$K_\chi$, acting on a vector space $H_\mu$. To such data we attach the unitary representation  \[\varrho_{\chi, \mu}=\Ind_{K_\chi \ltimes \mathfrak{p}}^{G_0}(\mu \otimes \chi)\] of~$G_0$, realized in its compact picture, namely the  closed subspace of $\operatorname{L^2}(K, H_\mu)$ defined by
\[ \Ind H_\mu=\left\{\ f \in \operatorname{L^2}(K, H_\mu) \ : \ f(km)=\mu(m)k \ \text{ for all  $m \in K_\chi$ and $k \in K$}\ \right\},\] with the action of $G_0$ defined by $\varrho_{\chi, \mu}(k, v) f = u \longmapsto e^{i \langle k \cdot \chi,  u\rangle} f(k^{-1} u)$.

\paragraph{The spectral extended quotient $\mathfrak{p}^\ast\doublequo K$ and its topology.} 
Let $\mathscr{D}$ be the set of all pairs $(\chi, \mu)$ where $\chi$ is an element of $\mathfrak{p}^\ast$ and $\mu$ is an \emph{irreducible} unitary representation of the stabilizer $K_\chi$. 

Define an equivalence relation on $\mathscr{D}$ by: $(\chi, \mu) \sim (\chi', \mu')$ if and only if there exists $k \in K$ such that, first, we have $\chi' = \chi \circ \operatorname{ad}(k)$, and second, the representation $\mu'$ of $K_{\chi'}$ on~$H_{\mu'}$ is equivalent to the representation $u \longmapsto \mu(k^{-1}uk)$ of $K_{\chi'}$ on~$H_\mu$. 

Let $\mathfrak{p}^\ast \doublequo K$ be the quotient $\mathscr{D}/\!\!\sim$: it is the \emph{spectral extended quotient} of Baum and Connes, used by Aubert, Baum, Plymen and Solleveld under various names (see e.g. \cite[\S~11]{ABPS_Takagi}). 

 For our purposes, it is important to recall that every $K$-orbit in $\mathfrak{p}$ meets $\amin$, and that the intersection is a single $W$-orbit; see \cite[Lemma 7.22 and Proposition 7.29]{Knapp_Beyond}. Since $\amin^\ast/W$ identifies with $\adomstar$ as in~\eqref{eq-chambre-dom}, it follows that every $K$-orbit in $\mathfrak{p}^\ast$  meets $\adomstar$ in a unique point. If  $\mathscr{D}_{\adomstar}$  denotes the subset of $\mathscr{D}$ consisting of all pairs $(\chi, \mu)$ with $\chi \in \adomstar$, we deduce that the inclusion $\mathscr{D}_{\adomstar} \hookrightarrow \mathscr{D}$  induces a bijection between $\mathscr{D}_{\adomstar}$ and $\mathfrak{p}^\ast \doublequo K$.
 
\smallskip

We shall use this to define a topology on $\mathfrak{p}^\ast \doublequo K$, building on the following observation.

\begin{lemma}\label{lem-taille-stab} Let $\chi$ be an element of $\amin^\ast$. For every $\chi' \in \amin^\ast$ that is close enough to~$\chi$, the stabilizer~$K_{\chi'}$ is contained in $K_\chi$. 
\end{lemma}
\begin{proof}
The centralizers $L_\chi$ and $L_\chi'$ of $\chi,\chi'$ in~$G$ are generated by $\Amin$ and by the root subgroups for those roots of $(\mathfrak{g}, \amin)$ that are orthogonal to $\chi$ and $\chi'$ respectively: see \cite[proof of Theorem~3.5]{Mackey_bijection}. If $\chi'$ is close enough to $\chi$, then a root that is not orthogonal to $\chi$ is not orthogonal to  $\chi'$, either; therefore $L_{\chi'} \subset L_\chi$, and $K_{\chi'} = L_{\chi'} \cap K$ is contained in $K_{\chi}=L_\chi \cap K$. 
\end{proof}
As an immediate consequence of the Lemma, for every $(\chi, \mu) \in \mathscr{D}_{\adomstar}$, there exists $\varepsilon_{\chi, \mu}>0$ with the property that if $\chi' \in \adomstar $ satisfies $\lVert \chi-\chi'\rVert < \varepsilon_{\chi, \mu}$ then the stabilizer $K_{\chi'}$ contains~$K_\chi$. We equip $\mathscr{D}_{\adomstar}$ with the topology which is generated, in the sense explained e.g. in~\cite[Chapter~I, \S~2, n°3, Exemple~II]{TG1}, by the sets 
\[ \mathscr{O}_{\chi, \mu}^{\varepsilon} = \left\{\ (\chi', \mu) \in \mathscr{D}_{\adomstar} \ : \ \lVert \chi-\chi'\rVert < \varepsilon \text{ and } \operatorname{Hom}_{K_{\chi'}}(\mu', \mu) \neq 0\ \right\}\]
for $0 < \varepsilon < \varepsilon_{\chi, \mu}$.
Transferring this topology using the bijection $\mathscr{D}_{\adomstar} \leftrightarrow \mathfrak{p}^\ast \doublequo K$ above, we obtain a topology on the spectral extended quotient $\mathfrak{p}^\ast \doublequo K$. 

\paragraph*{Classification of the unitary irreducible representations.} The following statement identifies the unitary dual of $G_0$, equipped with the Fell topology, and the spectral extended quotient $\mathfrak{p}^\ast\doublequo K$, equipped with the above topology.

\begin{theorem}[Unitary dual of $G_0$]\label{th-mackey-G0}~
\begin{enumerate}[(a)]
    \item For every $(\chi, \mu)$ in $\mathscr{D}$, the unitary representation $\varrho_{\chi, \mu}$ of~$G_0$ is irreducible.
    \item The map $\mathscr{D}\longrightarrow \hatGzero$ sending a pair $(\chi, \mu)$ to the equivalence class of $\varrho_{\chi, \mu}$ is surjective and descends to a bijection 
    \begin{equation}\label{eq-mackey-parametrization} \mathfrak{p}^\ast \doublequo K \overset{\simeq}{\longleftrightarrow} \hatGzero. \end{equation}
    \item The bijection~\textup{\eqref{eq-mackey-parametrization}} is a homeomorphism.
\end{enumerate}
\end{theorem}

Part~(a)-(b) are due to Mackey~\cite{Mackey1949}; part~(c) is due to Baggett \cite{Baggett}, although we have phrased it a little differently---our formulation of~(c) can be extracted e.g. from the survey of Echeterhoff and Emerson \cite{Echterhoff_Emerson} on the dual spaces of crossed products in an operator-algebraic context.

\paragraph*{Reformulation using the imaginary part of the infinitesimal character.} 
We shall now rephrase the classification in Theorem~\ref{th-mackey-G0} in a manner parallel to that in Section~\ref{sec-parametrisation-bhy}. 

First, let us define an analogue,
for irreducible representations of $G_0$,
of the imaginary part of the infinitesimal character. 
The projection
$(\chi, \mu) \longmapsto \chi$, from $\mathscr{D}$ to $ \mathfrak{p}^\ast$,
descends to a  surjection 
\[ \operatorname{pr}\colon \mathfrak{p}^\ast \doublequo K\longrightarrow \mathfrak{p}^\ast\!/K\]
from the spectral extended quotient $\mathfrak{p}^\ast \doublequo K$ to the ordinary quotient $\mathfrak{p}^\ast\!/K$.
Now, as mentioned in the discussion leading up to Lemma~\ref{lem-taille-stab} above, every  $K$-orbit in $\mathfrak{p}^\ast$ meets $\amin^\ast$ in a single $W$-orbit, and therefore meets $\adomstar$ in a single point. Combining these observations, we obtain a canonical surjection 
\[\mathfrak{p}^\ast \doublequo K \overset{\operatorname{pr}}{\longrightarrow} \mathfrak{p}^\ast\!/K \simeq \amin^\ast/W \overset{\simeq}{\longrightarrow} \adomstar,\]
where the final bijection is given by~\eqref{eq-chambre-dom}. 
Therefore Theorem~\ref{th-mackey-G0} determines a canonical surjection 
\[ \Iminfch\colon \hatGzero\longrightarrow \adomstar.\]
Given an irreducible representation $\varrho$ of $G_0$, we say that the corresponding element of $\adomstar$ is the \emph{imaginary part of the infinitesimal character of~$\varrho$}. For the connection with the usual notion of infinitesimal character, and for remarks on the absence of a ``real part of the infinitesimal character'' in contrast to the situation for reductive groups, see \cite[Section~4]{BHY}.

As in Section~\ref{sec-parametrisation-bhy} this leads to a partition of $\hatGzero$ according to the singularity of the imaginary part of the infinitesimal character. Recall the partition~\eqref{eq-partition-chambre} of $\adomstar$ into  subsets~$\aiplus$ indexed by  subsets $I \subset S$ of the set of simple roots. Given a subset $I$ of~$S$, a unitary representation $\mu$ of the compact group $K_I = M_I \cap K$, and an element $\nu$ of $\aiplus$, we shall denote by 
\begin{equation}\label{eq-def-rho-imunu} \varrho_{I, \mu, \nu} \colon G_0 \longrightarrow \operatorname{End}\left(\Ind\  H_\mu\right)\end{equation}
 the representation introduced as $\varrho_{\nu^\ast, \mu}$ above, for the unique element $\nu^\ast$ of $\mathfrak{p}^\ast$ which extends~$\nu$ and is zero on the orthogonal complement of $\amin$.

For later purposes, we mention that in \eqref{eq-def-rho-imunu} the unitary representation $\mu$ of $K_I$ is allowed to be reducible. But the case where $\mu$ is irreducible is of course important, and we shall often use the following reformulation of parts~(a) and (b) in Theorem~\ref{th-mackey-G0}.

\begin{theorembis}{th-mackey-G0}\label{th-mackey-G0-reformule}~
\begin{enumerate}[(a)]
\item For every triple $(I, \mu, \nu)$ where $I \subset S$, where $\mu$ is an irreducible unitary representation of~$K_I$, and where $\nu \in \aiplus$,  the representation $\varrho_{I, \mu, \nu}$ of~$G_0$ is irreducible.
\item Every irreducible unitary representation of~$G_0$ is equivalent with one of the~$\varrho_{I, \mu, \nu}$. 
\item Given triples $(I, \mu, \nu)$ and $(I', \mu', \nu')$ where $I, I' \subset S$, where $\mu$, $\mu'$ are unitary irreducible representations of $K_I$ and $K_{I'}$ respectively, and  where $\nu \in \aiplus$, $\nu' \in \mathfrak{a}^\ast_{I', +}$, we have the equivalence  $\varrho_{I, \mu, \nu} \simeq \varrho_{I', \mu', \nu'}$ if and only if $I=I'$, $\mu \simeq \mu'$ and $\nu=\nu'$.
\end{enumerate}
\end{theorembis}

\subsubsection{The Mackey bijection}\label{sec-mackey-bij}

With the above preparations, it is easy to define the Mackey bijection $\hatGzero \longrightarrow \hatGtemp$. 

It follows from Theorem~\ref{th-mackey-G0-reformule} that every unitary irreducible representation of~$G_0$ is equivalent to a representation $\varrho_{I, \mu, \nu}$ for an essentially unique triple $(I, \mu, \nu)$. On the other hand, it follows from Theorem~\ref{th-vogan-tempiric} that every irreducible representation $\mu$ of $K_I$ arises as the unique lowest $K_I$-type of a tempiric  representation $\tau(\mu)$ of~$M_I$, and the equivalence class of $\tau(\mu)$ is determined by that of $\mu$.

The following theorem, originally established in \cite{Mackey_bijection} with a slightly different formulation, is now an immediate consequence of  Theorems~\ref{th-bhy-classification} and \ref{th-mackey-G0-reformule}:

\begin{theorem}[Mackey bijection]
The assignment 
\[ \varrho_{I, \mu, \nu} \longmapsto \pi_{I, \tau(\mu), \nu} \]
determines a bijection $\mathscr{M}\colon \hatGzero \overset{\simeq}{\longrightarrow} \hatGtemp$.
\end{theorem}

It is clear that $\mathscr{M}$ preserves the imaginary part of the infinitesimal character, in the sense that $\Iminfch(\mathscr{M}(\pi))=\Iminfch(\pi)$ for every $\pi \in \hatGzero$. It is less obvious, but no less true, that it preserves the lowest $K$-types of representations: as in Section~\ref{sec-basics_LKT}, to any element $\pi$ of $\hatGzero$ we associate the finite subset $\LKT_{G_0}(\pi) \subset \Khat$ of lowest $K$-types of~$\pi$; then \cite[Proposition~4.1]{Mackey_bijection} says that 
\begin{equation}\label{eq-mackey-preserves-lkt}\LKT_G\big(\mathscr{M}(\pi)\big)=\LKT_{G_0}(\pi)\end{equation}
for every $\pi \in \hatGzero$.

\subsection{\texorpdfstring{The Mackey embedding of $\Cs$-algebras}{The Mackey embedding of C*-algebras}}
\label{sec-def-embedding}

For this section we need to interpret the (tempered) representations of $G$ and $G_0$ as representations of  certain $\Cs$-algebras. We shall refer to  Dixmier \cite[Chapters 13, 18]{Dixmier} and Bekka, de la Harpe and Valette \cite[Appendix~F]{BdlHV} for the details of the connection. For our purposes it will suffice to say that every unitary  (resp. tempered) representation~$\pi$ of $G_0$ (resp. $G$), acting on a complex Hilbert space~$H$, can be ``integrated''  to a nondegenerate representation of the universal $\Cs$-algebra $\Cs(G_0)$ (resp. of the reduced $\Cs$-algebra $\Csr(G)$), acting on the \emph{same} Hilbert space~$H$. We shall still denote by $\pi$ the corresponding $\Cs$-morphism from $\Cs(G_0)$ (resp. $\Csr(G)$) to the $\Cs$-algebra $\bounded(H)$ of all continuous linear endomorphisms of~$H$. This procedure preserves irreducibility, unitary equivalence, and produces all nondegenerate representations of $\Cs(G_0)$ or $\Csr(G)$. In this way the \emph{spectra} of $\Cs(G_0)$ and  $\Csr(G)$ (i.e. the sets of equivalence classes of irreducible representations) are identified with $\hatGzero$ and $\hatGtemp$, respectively. We shall often make such identifications without further comment.

Perhaps the main result of \cite{CHR} is the construction of a natural embedding of $\Cs(G_0)$ into $\Csr(G)$, as follows. 

 \begin{theorem}\label{th-chr} There exists an injective $\Cs$-algebra morphism 
\[ \alpha\colon \Cs(G_0) \longrightarrow \Csr(G)\]
such that for every $I \subset S$, for each tempiric representation $\tau$ of $M_I$, and for every $\nu \in \mathfrak{a}_I^\ast$, we have
\begin{equation} \label{eq-spectral-CHR} \pi_{I,\tau, \nu} \circ \alpha \simeq \varrho_{I,\tau, \nu}.\end{equation}
\end{theorem}

\begin{remark}\label{rk-restriction-ktypes}~
\begin{enumerate}[(a)]
\item
On the right-hand side of~\eqref{eq-spectral-CHR} is the unitary representation of $G_0$ attached to $I$, $\nu$ and the restriction of $\tau$ to $K_I$. In general this restriction is an infinite direct sum of irreducibles, and $\varrho_{I, \tau, \nu}$ is an infinite direct sum of irreducible representations of~$G_0$. A key point is that if  $\mu$ denotes the lowest $K_I$-type of $\tau$, then 
\begin{enumerate}[(i)]
\item $\varrho_{I, \tau, \nu}$ contains $\varrho_{I, \mu, \nu}$ with multiplicity one, 
\item the set of lowest $K$-types in $\varrho_{I, \tau, \nu}$ is exactly $\LKT\left(\pi_{I, \tau, \nu}\right)$, and
\item  these lowest $K$-types all appear in the irreducible summand $\varrho_{I, \mu, \nu}$ of $\varrho_{I, \tau, \nu}$. 
\end{enumerate}
For all this see \cite[proof of Proposition~4.1]{Mackey_bijection} and \cite[Lemma 7.4.2]{CHR}. 

\item Point~(a) above yields the following characterization of the  Mackey bijection: for every irreducible tempered representation $\pi$ of~$G$, or equivalently of $\Csr(G)$,  the composition $\pi \circ \alpha$ splits into a direct sum of irreducible representations of $\Cs(G_0)$,  the lowest $K$-types are all contained in the same irreducible summand, and the equivalence class of this summand is the element of $\hatGzero$ which corresponds to $\pi$ in the Mackey bijection.
\end{enumerate}
\end{remark}

\begin{remark} \label{rk-chr-constructif}
\begin{enumerate}[(a)]
\item While we formulated Theorem~\ref{th-chr} as an abstract existence result, in \cite{CHR} an embedding $\alpha\colon \Cs(G_0) \longrightarrow \Csr(G)$ as in the theorem is  constructed in a rather concrete manner. The construction depends on a number of choices which we shall discuss further  in Section~\ref{sec-localizations}. 

\item The representations of $\Cs(G_0)$ on both sides of~\eqref{eq-spectral-CHR} are on the same Hilbert space, viz. the space $\Ind H_\tau$ in \eqref{eq-def-Htau}. The statement is that these are unitarily equivalent, and \emph{not} equal: this plays a minor role in Sections~\ref{sec-strat-eq} and~\ref{sec-topology}, but will become essential in Section~\ref{sec-localizations}.
\item Elaborating on part~(a), let $I$ be a subset of $S$ and  $\tau$ be a tempiric representation of~$M_I$. Part of the theorem says that there exists a family $\left(\mathscr{B}_{\tau, \nu}\right)_{\nu \in \mathfrak{a}_I^\ast}$ 
of automorphisms of $\Ind H_\tau$ with the property that $\pi_{I, \tau, \nu}\left(\alpha(f)\right)=\mathscr{B}_{\tau, \nu}\,\varrho_{I, \tau, \nu}(f)\,\mathscr{B}_{\tau, \nu}^{-1}$
for all $\nu \in \mathfrak{a}_I^\ast$ and all $f \in \Cs(G_0)$. The embedding $\alpha\colon \Cs(G_0) \longrightarrow \Csr(G)$ constructed in \cite{CHR} comes with a concrete construction of a family $\left(\mathscr{B}_{\tau, \nu}\right)_{\nu \in \mathfrak{a}_I^\ast}$  of unitary operators that satisfy this property and which, moreover, depend \emph{continuously} on $\nu$. See \cite[Theorem 6.3.7]{CHR}. 
\end{enumerate}\end{remark}

Throughout the rest of the paper we fix an embedding $\alpha\colon \Cs(G_0)\longrightarrow \Csr(G)$ with the property in Theorem~\ref{th-chr}. In Section~\ref{sec-localizations} we shall moreover assume that it is obtained from the construction in \cite{CHR}, yielding continuous families $\left(\mathscr{B}_{\tau, \nu}\right)_{\nu \in \mathfrak{a}_I^\ast}$ of intertwining operators as in Remark~\ref{rk-chr-constructif}(c). See Section~\ref{sec-compatibility}, as well as Section~\ref{sec-converse-and-remarks} for related remarks on the \mbox{(non-)uniqueness} of an embedding $\alpha$ that satisfies~\eqref{eq-spectral-CHR}. 

\section{The Mackey embedding as a stratified equivalence}\label{sec-strat-eq}

\subsection{Basics on the Jacobson and Fell topologies}\label{sec-basics-cstar}
In Section~\ref{sec-def-embedding} we recalled that the duals $\hatGtemp$ and $\hatGzero$ can be identified with the spectra of the reduced $\Cs$-algebra $\Csr(G)$ and the universal $\Cs$-algebra $\Cs(G_0)$, respectively. 

In this section we collect basic facts on the interplay between the structure of a $\Cs$-algebra~$A$ and the Jacobson topology on its spectrum. We will use these when~$A$ is either $\Cs(G_0)$ or~$\Csr(G)$, in order to view the locally closed subsets of $\hatGzero$ or $\hatGtemp$ as the spectra of certain subquotients of $\Cs(G_0)$ or $\Csr(G)$. As an application we shall outline a proof that containment of a given $K$-type defines an open condition in the unitary dual of~$G$ or~$G_0$.  Readers who already know these facts  can safely skip this section. 

\subsubsection{Primitive ideals and the Fell topology}  \label{sec-basics-cstar-1}

Let $A$ be a $\Cs$-algebra. By the \emph{spectrum} of~$A$ we shall mean the set of equivalence classes of irreducible representations of~$A$, as in \cite{Dixmier}; we denote it by $\widehat{A}$ or $\operatorname{Spec}(A)$. By a \emph{primitive ideal} of~$A$, we shall mean the annihilator $\{ f \in A \ : \ \pi(f)=0\}$ of an irreducible representation~$\pi$ of~$A$; this annihilator depends only on the equivalence class of the representation~$\pi$.  We shall denote by $\Prim(A)$ the set of all primitive ideals, as in \cite[\S\,2.9.7]{Dixmier}, and equip it with the Jacobson topology \cite[\S\,3.1]{Dixmier}. We shall identify $\Prim(A)$ with $\widehat{A}$ as in \cite[loc. cit.]{Dixmier}.

By definition of the Jacobson topology, the closed subsets of $\widehat{A}$ then correspond to closed (two-sided) ideals in~$A$: if $X \subset \widehat{A}$ is closed, then  
\[ J_X = \{\ f \in A \ : \ \forall \pi \in X, \pi(f)=0\ \}\] is a closed ideal in~$A$, and every closed ideal arises in this way. 

Given a closed ideal $J$ in~$A$, every irreducible representation of~$J$ extends uniquely to an irreducible representation of~$A$, and this determines a homeomorphism between $\widehat{J}$ and the open subset $\widehat{A}^J$ of representations that do not vanish on~$J$: see \cite[\S\,2.11 and \S\,3.1]{Dixmier}. In this way, if $J \subset J' \subset A$ are nested closed ideals in~$A$, then the spectra of~$J$ and $J'$ correspond to nested open subsets in the spectrum of~$A$, and we shall identify the spectrum of $J'/J$ with the locally closed subset of $\widehat{A}$ given by the set-theoretic difference of these two open subsets. If~$\pi$~is a representation of~$A$ whose equivalence class belongs to that subset, we shall also use the letter $\pi$ to denote its restriction to~$J'$, or the representation of~$J'/J$ obtained from the latter by passing to the quotient.

\subsubsection{\texorpdfstring{Containment of $K$-types as an open condition}{Containment of K-types as an open condition}}

Recall that if~$G$ is a locally compact group, the unitary dual $\widehat{G}$ identifies with the spectrum of the universal $\Cs$-algebra $\Cs(G)$; this identifies  the Fell topology on $\widehat{G}$ with the Jacobson topology on the spectrum of $\Cs(G)$.  As an illustration of the ideas in Section~\ref{sec-basics-cstar-1}, we indicate an application to unitary group representations which is widely known, but perhaps not that easy to locate in the literature, and important for our later purposes. 

\begin{lemma}\label{lem-Ktype-ouvert}
Let $G$ be a locally compact group, and let $K$ be a compact subgroup of~$G$. Denote the unitary duals of $G$ and $K$ by $\widehat{G}$ and $\widehat{K}$. 
 Let $\lambda$ be an element of $\Khat$, and let $\mathcal{U}_\lambda$  be the subset of~$\widehat{G}$ consisting of those representations $\pi$ whose restriction to~$K$ contains~$\lambda$. Then $\mathcal{U}_\lambda$ is open in $\widehat{G}$. 
\end{lemma}
\begin{proof}  Let $\chi_\lambda$  be the character of~$\lambda$, viewed as a continuous function $K \longrightarrow \mathbb{C}$. Then, once a left Haar measure on~$G$ is fixed, convolution by $\chi_\lambda$ defines a multiplier $p_{\lambda}$ of the $\Cs$-algebra~$\Cs(G)$; one may then attach to $\lambda$ the closed ideal $J_\lambda=\Cs(G) p_{\lambda} \Cs(G)$.  An irreducible unitary representation of~$G$ vanishes on $J_\lambda$ if and only if it does not contain $\lambda$ upon restriction to~$K$. Since $J_\lambda$ is a closed ideal in $\Cs(G)$, it follows that the irreducible unitary representations which do not contain $\lambda$ make up a closed subset of~$\widehat{G}$. Since this is the complement of $\mathcal{U}_\lambda$ in  $\widehat{G}$, the result follows.\end{proof}

\subsection{\texorpdfstring{Stratified equivalence morphisms for topological $\mathbb{C}$-algebras}{Stratified equivalence morphisms for topological C-algebras}}\label{sec-strat-defs}

This subsection collects definitions that, up to minor variation, arose from the study of certain Hecke algebras attached to Bernstein components of the smooth dual of $p$-adic groups: for this see the work of Lusztig \cite{Lusztig}, Baum--Nistor \cite{BaumNistor}, and Aubert--Baum--Plymen--Solleveld (especially \cite{ABPS_Morita}).

\begin{definition}[Baum--Nistor]\label{def-baumnistor} Let~$A$ and~$B$ be $\mathbb{C}$-algebras, and $f\colon A \longrightarrow B$ be an algebra homomorphism. The morphism~$f$ is \emph{spectrum-preserving} if it satisfies the following two conditions:
\begin{enumerate}[(i)]
    \item For every primitive ideal~$J\subset B$, the set $f^{-1}(J)$ is contained in a unique primitive ideal of~$A$. 
    \item The resulting map $\Prim(B) \longrightarrow \Prim(A)$, which sends any primitive ideal~$J$ of~$B$ to the unique primitive ideal of~$A$ that contains $f^{-1}(J)$, is a bijection. 
\end{enumerate}
\end{definition}

\begin{remark}\label{rem-finite-type} If $A$ and~$B$ are finitely generated as algebras, then Baum and Nistor proved that the map $\operatorname{Prim}(B) \longrightarrow \operatorname{Prim}(A)$ attached to $f$ is automatically a homeomorphism in the Jacobson topology \cite[Theorem~2]{BaumNistor}.
\end{remark}

Now suppose~$A$ is a   topological $\mathbb{C}$-algebra. By a \emph{filtration} of~$A$ we shall mean an increasing family $\mathcal{J}=(J_n)_{n \in \N}$ of closed ideals, with $J_0=\{0\}$, such that the union $\bigcup_{n \in \N} J_n$ is dense in~$A$. If there exists $n \in \N$ such that $J_n=A$, we shall say the filtration $\mathcal{J}$ is \emph{finite}, and call \emph{number of pieces of~$\mathcal{J}$} the smallest integer $n\geq 1$ such that $J_n=A$. 

\begin{definition} \label{def-stratified-equivalence} Let~$A$ and~$B$ be topological $\mathbb{C}$-algebras, and let~$f\colon A \longrightarrow B$ be a continuous algebra homomorphism. Let $\mathcal{I}=\left(I_n\right)_{n \in \N}$ and $\mathcal{J}=(J_n)_{n \in \N}$ be filtrations of~$A$ and~$B$, respectively. We call $f$ a \emph{stratified equivalence morphism} (with respect to the filtrations $\mathcal{I}$ and~$\mathcal{J}$) when:
\begin{enumerate}[(i)]
    \item The morphism $f$ respects the filtrations $\mathcal{I}$ and~$\mathcal{J}$, i.e. $f\left(I_n\right) \subset J_n$ for all $n$; 
    \item For all $n \geq 1$, the morphism $I_n/I_{n-1} \longrightarrow J_n/J_{n-1}$ induced by $f$ is spectrum-preserving. 
\end{enumerate}
\end{definition}

\begin{remark}\label{rk-stratified-induces-bijection}
    Any stratified equivalence morphism $f\colon A \longrightarrow B$ gives rise to a bijection between the spectrum of~$B$ and the spectrum of~$A$, since these can be identified with the disjoint unions of the spectra of all subquotients $B_{i}/B_{i-1}$ and $A_i/A_{i-1}$. By Remark~\ref{rem-finite-type}, if~$A$ and~$B$ are finitely generated as $\mathbb{C}$-algebras and if the given filtrations on $A$ and~$B$ are finite (with the same number of pieces), then this bijection is a \emph{piecewise homeomorphism}: in the notation of Definition~\ref{def-stratified-equivalence}, for each integer~$n$, it induces a homeomorphism between the subsets of $\operatorname{Prim}(B)$ and $\operatorname{Prim}(A)$ corresponding to $J_n/J_{n-1}$ and $I_n/I_{n-1}$.
\end{remark}

In the rest of this paper, we shall use the notions of spectrum-preserving and stratified equivalence morphisms in the context of $\Cs$-algebra morphisms. These are of course also homomorphisms of topological $\mathbb{C}$-algebras, but in the $\Cs$-algebraic context the notion of primitive ideal is that of  Section~\ref{sec-basics-cstar}, and therefore slightly different from the more algebraic version used in Definition~\ref{def-baumnistor}. However, using the $\Cs$-algebraic notion of primitive ideal, Definitions~\ref{def-baumnistor} and~\ref{def-stratified-equivalence} apply verbatim, and yield the notions of spectrum-preserving and stratified equivalence morphisms which we shall use in the rest of this paper.

\subsection{\texorpdfstring{Two stratifications of $\Cs(G_0)$ and $\Csr(G)$ by lowest $K$-types}{Two stratifications of C*(G0) and C*r(G) by lowest K-types}}\label{sec-filtrations}

\subsubsection{\texorpdfstring{Stratification by the height of the lowest $K$-types}{Stratification by the height of the lowest K-types}}\label{sec-first-filtration}

Recall that on each of $\hatGtemp$ and $\hatGzero$, we defined a ``height of the lowest $K$-types''  function
\begin{align*}
\pi & \longmapsto \Knorm{\LKT(\pi)}
\end{align*}
in Section~\ref{sec-basics_LKT}. The height functions attached to $\hatGtemp$ and $\hatGzero$ have the same image in~$\R_{\geq 0}$. It is a discrete, locally finite set~$\Lambda$ of real numbers, and we shall write
\begin{equation} \Lambda = \{ R_1, \ldots, R_k, \ldots \} \end{equation}
where $(R_k)_{k \geq 1}$ is a strictly increasing sequence of nonnegative real numbers. To simplify our later notation, we set $R_0=-1$ (any negative real number would do). 

We shall define filtrations of $\Cs(G_0)$ and $\Csr(G)$ according to the height of  lowest $K$-types. For any nonnegative integer~$k$, let ${\mathscr{V}}_k$ be following the subset of $\hatGtemp$: 
\begin{equation}\label{eq-def-vk} 
\mathscr{V}_k=\left\{\ \pi \in \hatGtemp \ : \ \Knorm{\LKT(\pi)} > R_k\ \right\}.\end{equation}  Consider its annihilator ${\Ups}_k$ in~$\Csr(G)$: 
\[{\Ups}_k = \left\{\ f \in \Csr(G)\ : \ \forall \pi \in {\mathscr{V}}_{k}, \  \pi(f)=0 \ \right\}.\]

Similarly, on the $G_0$-side let ${\mathscr{V}}_k^0$ be the subset of $\hatGzero$ which comprises the irreducible representations~$\pi$ such that $\Knorm{\LKT(\pi)} > R_k$, and let ${\Ups}^0_k$ be the annihilator of $\mathcal{V}_k^0$ in~$\Cs(G_0)$.

\begin{proposition} \label{prop-first-filtration}

\begin{enumerate}[(a)]
\item For every nonnegative integer $k$, the collection $\mathscr{V}_k$  of representations is a closed subset of $\hatGtemp$, and ${\Ups}_k$  is a closed ideal in $\Csr(G)$. 
\item For every nonnegative integer $k$, the set $\mathscr{V}^0_k$  is closed in $\hatGzero$, and ${\Ups}^0_k$  is a closed ideal in $\Cs(G_0)$. 
\item The families $(\Ups_k)_{k \in \N}$ and $(\Ups^0_k)_{k \in \N}$ are increasing filtrations of $\Csr(G)$ and $\Cs(G_0)$, respectively.
\end{enumerate}
\end{proposition}

\begin{proof}
By definition, the set $\mathcal{V}_k$ consists of those irreducible tempered representations whose restriction to~$K$ contains none of the representations $\lambda \in \Khat$ with $\Knorm{\lambda} \leq R_k$. For $\lambda \in \Khat$, denote by $\mathcal{U}_\lambda$ the subset of $\hatGtemp$ consisting of those representations which contain the $K$-type $\lambda$. Thus $\mathcal{V}_k$ is the intersection of the sets $\hatGtemp \setminus \mathcal{U}_\lambda$, for $\lambda$ such that  $\Knorm{\lambda} \leq R_k$. By Lemma~\ref{lem-Ktype-ouvert}, combined with the fact that $\hatGtemp$ is closed in the unitary dual of~$G$, we know that the sets $\hatGtemp \setminus \mathcal{U}_\lambda$ are all closed. Part~(a) follows immediately.
The same argument applies with $G_0$ instead of~$G$ and yields~(b). For~(c), it is clear from the definitions that $\mathcal{V}_{k+1} \subset \mathcal{V}_k$ for all~$k$, therefore $\Ups_k \subset \Ups_{k+1}$ for all~$k$. 
Finally, the union of all sets $\mathcal{V}_k$ is all of $\hatGtemp$; therefore  the union of the $\Ups_k$ is dense in $\Csr(G)$. The same argument applies to prove that the union of the~$\Ups_k^0$ is dense in $\Cs(G_0)$.
\end{proof}

\subsubsection{\texorpdfstring{Stratification by the possible sets of lowest $K$-types}{Stratification by the possible sets of lowest K-types}}\label{sec-second-filtration}

In this section, we define another pair of filtrations of $\Cs(G_0)$ and $\Csr(G)$ which keep track of the precise sets of lowest $K$-types of representations, not just the height of the lowest $K$-types. As a result, these filtrations will be finer than the ones in Proposition~\ref{prop-first-filtration}.

Let $\mathscr{E}$ be the collection of subsets of $\Khat$ that arise as $\LKT(\pi)$ for at least one irreducible tempered representation~$\pi$ of~$G$. In the notation of \S~\ref{sec-basics_LKT}, on any given element~$E$ of~$\mathscr{E}$, the function $\lambda \longmapsto \Knorm{\lambda}$ is constant; in other words the image $\Knorm{E}$ consists of a single number.

The set $\mathscr{E}$ is countable; we shall fix a labelling 
\[ \mathscr{E} = \{E_1, E_2,  \ldots \}\]
satisfying the following conditions: 
\begin{enumerate}[(i)]
\item If $\Knorm{E_p} < \Knorm{E_q}$, then $p<q$;
\item If $\Knorm{E_p}=\Knorm{E_q}$ and $\operatorname{Card}(E_p)>\operatorname{Card}(E_q)$, then $p<q$. 
\end{enumerate}

\begin{example} 
To illustrate the role of~(ii), let us consider the example of $G = \mathrm{SL}(2, \R)$. For $j \in \mathbb{Z}$, let us write $\lambda_{j}$ for the one-dimensional representation of $K=\mathrm{SO}(2, \R)$ in which the rotation with angle $\theta$ acts by multiplication by $e^{ij\theta}$. Then $E_1 = \{\lambda_0\}$ is the set of lowest $K$-types of the spherical principal series, $E_2 = \{ \lambda_1, \lambda_{-1} \}$ is the set of lowest $K$-types for the nonspherical principal series with regular infinitesimal character, and $E_3$ and $E_4$ are each the set of lowest $K$-type for a limit of discrete series representation. So each of $E_3$ and $E_4$ contains   just one element,  either $\lambda_1$ or $\lambda_{-1}$, but (i)-(ii) do not determine which is $\{\lambda_1\}$ and which is $\{\lambda_{-1}\}$. We see that the numbering of $\mathscr{E}$ partly helps keep track of the ``regularity'' of the imaginary part of the infinitesimal character.
\end{example}

\begin{remark} \label{remark-ktypes-ordering} An immediate consequence of~(ii) is that  for fixed $n \geq 2$, the~set $E_n$ cannot contain any of the sets $E_j$ for $j < n$.\end{remark}

\noindent For any positive integer~$n$, let $\widetilde{\mathscr{W}}_n$ be the following subset of $\hatGtemp$: 
\begin{equation*}\widetilde{\mathscr{W}}_n = \left\{\ \pi \in \hatGtemp \ : \ \pi_{|K} \not\supset E_n\ \right\}.\end{equation*}
Consider its annihilator~$J_n$ in~$\Csr(G)$: 
\[J_n = \left\{\ f \in \Csr(G)\ : \ \forall \pi \in \widetilde{\mathscr{W}}_{n}, \  \pi(f)=0\ \right\}.\]
This an ideal of~$\Csr(G)$: it is the ideal denoted by $\mathbf{J}[n]$ in \cite[Section~3.1.1]{Mackey_CK_isom}. Accordingly  $\widetilde{\mathscr{W}}_n$ is a closed subset of $\hatGtemp$. We set $J_{0}=0$ and $\widetilde{\mathscr{W}}_0 = \hatGtemp$.

For any nonnegative integer~$n$, we now set 
\[ {I}_n  = J_0 + \cdots + J_n;\]
this is the ideal denoted by $\mathbf{I}[n]$ in \cite[loc. cit.]{Mackey_CK_isom}.
 In representation-theoretic terms,  set 
 \begin{equation}  \label{eq-def-wn}\mathscr{W}_n= \bigcap \limits_{j \leq n}\widetilde{\mathscr{W}}_j.\end{equation}
 This is  the closed set of those $\pi \in \hatGtemp$ such that $\pi_{|K}$ does not contain any of the sets $E_j$ for $j \leq n$; or equivalently, such that $\LKT(\pi)$ does not contain any of the  $E_j$, $j \leq n$. Then $I_n$ is the annihilator of~$\mathscr{W}_n$:
\[ I_n =  \left\{\ f \in \Csr(G)\ : \ \forall \pi \in \mathscr{W}_n, \  \pi(f)=0\ \right\}.\]
The spectrum of $I_n$ is the open subset of $\hatGtemp$ whose elements are the irreducible tempered representations~$\pi$ such that $\LKT(\pi)$ contains  $E_j$ for some $j \leq n$. Therefore the union of the spectra of all ideals $I_n$ is all of $\hatGtemp$, and as a result the union of all ideals $I_n$ is dense in~$\Csr(G)$. Since it is clear that $I_n \subset I_{n+1}$ for all $n$, we see that $\left(I_n\right)_{n \in \N}$ is an increasing filtration of $\Csr(G)$. The subquotients of this filtration satisfy
\[ \operatorname{Spec}(I_n/I_{n-1}) = \LKT_G^{-1}(E_n), \quad \forall n \geq 1.\]
Switching to the Cartan motion group, define ideals in $\Cs(G_0)$  by setting $J_0^0=0$ and
\[J^0_n = \left\{\ f \in \Cs(G_0)\ : \quad \forall \pi \in \hatGzero \:,\: \left( \pi_{|K} \not\supset E_n \right) \implies \pi(f)=0\ \right\}\]
for $n \geq 1$; then set 
\[ I_n^0 =J_0^0+\cdots +  J_n^0\]
for $n \geq 0$. This is an ideal of $\Cs(G_0)$, and we  define $\mathscr{W}_n^0$ to be the closed subset of $\hatGzero$ with annihilator $I_n^0$. The union of the sets $\mathcal{W}_n^0$ is all of $\hatGzero$, and  $(I_{n}^0)_{n \in \N}$ is an increasing filtration of $\Cs(G_0)$. 

\begin{remark} The filtrations $(I_n)_{n \in \N}$ and $(I_n^0)_{n \in \N}$ are a refinement of the filtrations $(\Ups_k)_{k \in \N}$ and $(\Ups^0_k)_{k \in \N}$ from the previous section: grouping the elements $E$ of $\mathscr{E}$ according to the value of~$\Knorm{E}$, it is clear that each of the ideals $\Ups_k$  can be written as $\sum_{n \leq N_k} I_k$ for a certain integer~$N_k$; and then we have $\Ups_{k}^0 = \sum_{n \leq N_k} I_k^0$. 
\end{remark}

\subsection{Proof of the stratified equivalence property}

 We shall prove that the Mackey embedding $\alpha\colon \Cs(G_0) \longrightarrow \Csr(G)$ is a stratified equivalence morphisms for the two sets of filtrations of $\Cs(G_0)$ and $\Csr(G)$ defined in Sections~\ref{sec-first-filtration} and~\ref{sec-second-filtration}. The proofs are very similar, but neither of the results for one set of filtrations implies the  result for the other set of filtrations. Therefore we shall give parallel statements and proofs.
 
 The first step is to prove that the Mackey embedding respects the two sets of filtrations:

\begin{proposition}\label{prop-compatibility-stratifications}~
\begin{enumerate}[(a)]
\item For every nonnegative integer $k$, we have $\alpha\left(\Ups^0_k\right) \subset \Ups_k$.
\item For every nonnegative integer $n$, we have $\alpha\left(I^0_n\right) \subset I_n$.
\end{enumerate}
\end{proposition}

This will be an easy consequence of the following more precise statement about the composition of any irreducible tempered representation $\pi$ of $G$ (viewed as a representation of $\Csr(G)$) with the Mackey embedding. The statement uses the notation of Section~\ref{sec-filtrations}, and in particular the families $(\mathcal{V}_k)_{k \in \N}$ and $(\mathcal{W}_n)_{n \in \N}$ of closed subsets of $\hatGtemp$ defined in~\eqref{eq-def-vk} and~\eqref{eq-def-wn}.

\begin{proposition}\label{prop-technical-compatibility}
Let $\pi$ be an irreducible tempered representation of~$G$, and let $\varrho$ be an irreducible representation of~$G_0$ whose equivalence class corresponds to~$\pi$ under the Mackey bijection. 
\begin{enumerate}[(a)]
    \item  Let $k \geq 0$ be an integer. Suppose the equivalence class of~$\pi$ belongs to~$\mathcal{V}_k$.  For every element $f$ in the ideal $\Ups_{k+1}^0$ of $\Cs(G_0)$, we have $\pi\left(\alpha(f)\right) = 0$ if and only if $\varrho(f)=0$. 
      \item  Let $n \geq 0$ be an integer. Suppose the equivalence class of~$\pi$ belongs to~$\mathcal{W}_n$. For every element $f$ in the ideal $I_{n+1}^0$ of $\Cs(G_0)$, we have $\pi\left(\alpha(f)\right) = 0$ if and only if $\varrho(f)=0$.  
\end{enumerate}  
\end{proposition}

\begin{proof}[Proof of Proposition~\ref{prop-compatibility-stratifications}, given Proposition~\ref{prop-technical-compatibility}] 
For~(a), fix an integer $k \geq 0$ and an element $f \in \Ups_{k}^0$. We want to prove that $\alpha(f) \in \Ups_k$, i.e. that if $\pi$ is an irreducible representation of~$G$ whose equivalence class belongs to $\mathcal{V}_k$, then $\pi\left(\alpha(f)\right)=0$. 

Fix such a representation~$\pi$, and let~$\varrho$ be an irreducible representation of~$G_0$ whose equivalence class corresponds to that of~$\pi$ under the Mackey bijection. Since the Mackey bijection preserves lowest $K$-types, we have $\LKT(\varrho)=\LKT(\pi)$, thus $\varrho \in \mathcal{V}_k^0$ by definition of $\mathcal{V}_k^0$, and therefore $\varrho(f)=0$ since $f$ belongs to the annihilator $\Ups_k^0$ of $\mathcal{V}^0_k$. But we also have $f \in \Ups^0_{k+1}$ since $\Ups_k^0 \subset \Ups_{k+1}^0 $; by Proposition~\ref{prop-technical-compatibility} we deduce $\pi\left(\alpha(f)\right)=0$, q.e.d. 

The proof of~(b) is exactly the same using the fact that $\mathcal{W}^0_n$ is mapped to $\mathcal{W}_n$ under the Mackey bijection, and the definitions of $I_n^0$ and $I_n$ as the annihilators of $\mathcal{W}_n^0$ and $\mathcal{W}_n$. 
\end{proof}

\begin{proof}[Proof of Proposition~\ref{prop-technical-compatibility}]

Let $\pi$ be an irreducible tempered representation of~$G$. Up to unitary equivalence we may assume $\pi=\pi_{I, \tau, \nu}$ where $I$ is a subset of~$S$, $\tau$ is a tempiric representation of~$M_I$, and $\nu$ is an element of $\aiplus$.

\begin{lemma}\label{lem-technical-length} Let $\mu$ be the lowest $K_I$-type of~$\tau$. For every irreducible representation~$\lambda$ of $K_I$ which occurs in $\tau|_{K_I}$ and is not equivalent to~$\mu$, the irreducible representation $\varrho_{I, \lambda, \nu}$ of~$G_0$ satisfies $\Knorm{\LKT(\varrho_{I, \lambda, \nu})} > \Knorm{\LKT(\pi)}$. 
\end{lemma}

\begin{proof}
 For every irreducible representation $\lambda$ of $K_I$, let us write  $\mult(\tau, \lambda)$ for the multiplicity with which $\lambda$ occurs in $\tau|_{K_I}$.  Let us consider the (possibly reducible) representation  $\varrho^\ast=\varrho_{I, \tau, \nu}$ of~$G_0$. It is immediate from the definitions that 
\begin{equation}\label{eq-dec-isotyp} \varrho^\ast \simeq \bigoplus_{\lambda \in \hatKI} \mult(\lambda, \tau) \cdot \varrho_{I,\lambda, \nu}.\end{equation}
By Remark~\ref{rk-restriction-ktypes}, the lowest $K$-types of $\pi$ and $\varrho^\ast$ are the same, and occur only in the summand corresponding to $\lambda = \mu$. Therefore if $\mult(\lambda, \tau) \neq 0$ and $\lambda \neq \mu$, the $K$-types that appear in~$\varrho_{I, \lambda, \mu}$ must have height strictly larger than $\Knorm{\LKT(\pi)}$, and the lemma follows. \end{proof}

To continue the proof of Proposition~\ref{prop-technical-compatibility}, let us consider the composition~$\pi \circ \alpha$ of $\pi$ with the Mackey embedding, viewed as a representation of~$\Cs(G_0)$. 
Then the spectral property~\eqref{eq-spectral-CHR} says $\pi \circ \alpha \simeq \varrho_{I, \tau, \nu}=\varrho^\ast$, and together with the decomposition~\eqref{eq-dec-isotyp} this gives
\begin{equation}\label{eq-compose-ups} \pi \circ \alpha \simeq \bigoplus_{\lambda \in \hatKI} \mult(\lambda, \tau) \cdot \varrho_{I,\lambda, \nu}.\end{equation}
For part~(a) let us assume that the equivalence class of $\pi$ belongs to $\mathcal{V}_k$, in other words, that $\Knorm{\LKT(\pi)} \geq R_{k+1}$. Then on the right-hand side of \eqref{eq-compose-ups}, whenever $\lambda \in \hatKI$ is such that $\mult(\lambda, \tau) \neq 0$ and $\lambda \neq \mu$, we must have $\Knorm{\LKT(\varrho_{I, \lambda, \nu})} > R_{k+1}$ by Lemma~\ref{lem-technical-length}; therefore $\Knorm{\LKT(\varrho_{I, \lambda, \nu})} \geq R_{k+2}$. By definition of~$\Ups_{k+1}^0$, we see that if $f \in \Ups_{k+1}^0$, then all nonzero summands on the right-hand side of \eqref{eq-compose-ups} must vanish on~$f$ except the summand corresponding to $\lambda=\mu$, which appears with multiplicity one, and is equivalent to~$\varrho$ by definition of the Mackey bijection. Therefore $(\pi \circ \alpha)(f)=0$ if and only if $\varrho(f)=0$, as claimed. 

For part~(b), let us assume that the equivalence class of $\pi$ belongs to $\mathcal{W}_n$, in other words, that $\LKT(\pi)$ contains none of the $E_j$ for $j \leq n$. Then $\LKT(\pi)$ must be $E_m$ for an integer $m > n$. Now
whenever $\lambda \in \hatKI$ is such that $\mult(\lambda, \tau) \neq 0$ and $\lambda \neq \mu$,  by Lemma~\ref{lem-technical-length} we must have $\Knorm{\LKT(\varrho_{I, \lambda, \nu})} > \Knorm{\LKT(\pi)}$.
Since the ordering of the sets $E_j$ is by increasing height, this implies that $\LKT(\varrho_{I, \lambda, \nu})$ can contain none of the $E_j$ for $j \leq m$; since $m >n$, this implies that all representations $\varrho_{I, \lambda, \nu}$ vanish on the ideal $I_{n+1}$ by definition. The conclusion follows as in~(a), and Proposition~\ref{prop-technical-compatibility} is proved. 
\end{proof}

We have now proved that the Mackey embedding $\alpha\colon \Cs(G_0) \longrightarrow \Csr(G)$ preserves the filtrations of Sections~\ref{sec-first-filtration} and~\ref{sec-second-filtration}, and turn to the proof that it induces spectrum-preserving morphisms on subquotients. For $k, n \geq 1$, write  
\begin{align*}
\alpha_k^\Ups\colon \Ups_k^0/\Ups_{k-1}^0 & \longrightarrow  \Ups_k/\Ups_{k-1}\\
\intertext{and}
\alpha_n\colon I_n^0/I_{n-1}^0 & \longrightarrow  I_n/I_{n-1}
\end{align*}
for the $\Cs$-morphisms induced by $\alpha$.
\bigskip
\begin{proposition}\label{prop-stratified-embedding}~
\begin{enumerate}[(a)]
    \item For every integer $k \geq 1$, the morphism $\alpha_k^\Ups$ is spectrum-preserving.
    \item For every integer $n \geq 1$, the morphism $\alpha_n$ is spectrum-preserving.
\end{enumerate}
\end{proposition}

Recall that the spectra of $\Ups_{k}^0/\Ups_{k-1}^0$ and $\Ups_{k}/\Ups_{k-1}$ can be identified with the subsets of $\hatGzero$ and $\hatGtemp$ comprising the representations whose lowest $K$-types have height exactly~$R_k$. Since the Mackey bijection preserves lowest $K$-types, it induces a bijection betwen these spectra. The same applies to the spectra of $I^0_n/I^0_{n-1}$ and $I_{n}/I_{n-1}$, which can be identified with the subsets of~$\hatGzero$ and $\hatGtemp$ comprising the representations whose set of lowest $K$-types is exactly~$E_n$.

The key fact to prove the proposition is: 

\begin{lemma}\label{lem-technical-strat}
    Let $n \geq 1$ be an integer and let $\pi$ be an irreducible tempered representation of~$G$ with $\LKT(\pi)=E_n$. 
    Let $\varrho$ be an irreducible unitary representation of~$G_0$ whose equivalence class corresponds to that of~$\pi$ under the Mackey bijection. 
    \begin{enumerate}[(a)]
    \item Let $V$ be the primitive ideal in~$I_n/I_{n-1}$ attached to~$\pi$, and let $V_0$ be the primitive ideal in~$I_n^{0}/I_{n-1}^0$ attached to~$\varrho$. We have the equality 
    $V_0 = \alpha_{n}^{-1}(V)$. 
    \item Let $k$ be the integer such that $\Knorm{E_n}=R_k$, let $V$ be the primitive ideal in $\Ups_k/\Ups_{k-1}$ attached to~$\pi$, and let~$V_0$ be the primitive ideal in $\Ups_{k}^0/\Ups_{k-1}^0$ attached to~$\varrho$.
    We have $V_0 = (\alpha_k^\Ups)^{-1}(V)$. 
    \end{enumerate}
\end{lemma}

 \begin{proof}
   For~(a), let $\varphi$ be an element of~$\Ups_{n}^0/\Ups_{n-1}^0$, and let $f \in \Ups_{n}^0$ be an element
which maps to~$\varphi$ in the projection $\Ups_n^0 \longrightarrow \Ups_n^0/\Ups_{n-1}^0$.
By definition $\LKT(\pi)$ is $E_n$,
and therefore it contains none of the sets $E_j$ for $j \leq n-1$
(see Remark~\ref{remark-ktypes-ordering});
in the notation of \S~\ref{sec-second-filtration} this means $\pi \in \mathscr{W}_{n-1}$.
Since $f \in I_n$, we may apply Proposition~\ref{prop-technical-compatibility}
and see that $ \pi\left(\alpha(f)\right)=0$ if and only if $ \varrho(f)=0.$
 In the notation of Section~\ref{sec-basics-cstar-1} concerning representations of subquotients, 
this means
\begin{equation*}
\pi(\alpha_n(\varphi))=0 
\iff
\varrho(\varphi)=0.
\end{equation*}
By definition~$V$ is the kernel of~$\pi$,
and $V_0$ is the kernel of $\varrho$;
this proves assertion~(a). 

The proof of~(b) is exactly the same using the fact that $\Knorm{\LKT(\pi)} = R_k > R_{k-1}$, and therefore $\pi \in \mathcal{V}_{k-1}$. This concludes the proof of Lemma~\ref{lem-technical-strat}.
 \end{proof}

Let us now prove Proposition~\ref{prop-stratified-embedding}. We shall prove~(b); part~(a) is obtained by the exact same argument.

The first thing to be checked is that in the notation of Lemma~\ref{lem-technical-strat},
the only primitive ideal of $\Ups_k^0/\Ups_{k-1}^0$
that contains ($\alpha^\Ups_k)^{-1}(V)$
is $V_0$.
Let $W_0$ be a primitive ideal in $\Ups_k^0/\Ups_{k-1}^0$ 
that contains the ideal $(\alpha^\Ups_k)^{-1}(V)$.
By Lemma~\ref{lem-technical-strat}  this means $V_0 \subset W_0$.
But the primitive ideals of~$\Ups^0_k/\Ups_{k-1}^0$
are all maximal among closed ideals:
this follows from  \cite[\S\,4.4.1]{Dixmier},
because the $\Cs$-algebra $\Ups^0_k/\Ups_{k-1}^0$
is liminal in the sense of \cite[\S\,4.3]{Dixmier}.
To see that it is liminal, recall that $\Ups^0_k/\Ups_{k-1}^0$ is a subquotient of $\Cs(G_0)$,
and every subquotient of a liminal $\Cs$-algebra
is again liminal \cite[Proposition 4.2.4]{Dixmier};
since  $\Cs(G_0)$ is liminal  --- which can be seen e.g. by combining \cite[Theorem 3.2]{Higson_08} and \cite[Examples 1.42]{Crisp} ---
this implies our statement.
Therefore the inclusion $V_0 \subset W_0$
can hold only if $V_0 = W_0$.

The final thing to be checked is that the map $\pi \longmapsto \varrho$, from the spectrum of $\Ups_{k}/\Ups_{k-1}$ to the spectrum of $\Ups^0_{k}/\Ups^0_{k-1}$, is bijective. But since  this map can be identified with  the restriction of the Mackey bijection to the subsets of $\hatGtemp$ and $\hatGzero$ where the lowest $K$-types of representations have height $R_k$, the result follows from  the fact that the Mackey bijection preserves lowest $K$-types. This concludes the proof.\qed \\

\newpage
\noindent We have obtained:

\begin{theorem}\label{th-strat-LKT}~
\begin{enumerate}[(a)]
\item The embedding $\alpha\colon \Cs(G_0) \longrightarrow \Csr(G)$
is a stratified equivalence morphism for the filtrations
$(\Ups^0_k)_{k \in \N}$ and $(\Ups_k)_{k \in \N}$.
\item The embedding $\alpha\colon \Cs(G_0) \longrightarrow \Csr(G)$
is a stratified equivalence morphism for the filtrations
$(I_n^0)_{n \in \N}$ and $\left(I_n\right)_{n \in \N}$.
\end{enumerate}

\end{theorem}

In both~(a) and~(b) the bijection $\hatGtemp \longrightarrow \hatGzero$ attached to $\alpha$ by Remark~\ref{rk-stratified-induces-bijection} is, of course, the Mackey bijection: this follows immediately from Lemma~\ref{lem-technical-strat}.

\section{Topological properties of the Mackey bijection}\label{sec-topology}

In this section, we prove that for all $k \geq 1$, the Mackey bijection induces a homeomorphism between the spectra of $\Ups_k^0/\Ups_{k-1}^0$ and that of $\Ups_k/\Ups_{k-1}$, and similarly that for all $n  \geq 1$, it induces a homeomorphism between the spectra of $I_n^0/I_{n-1}^0$ and $I_n/I_{n-1}$. 

If $\Cs(G_0)$ and $\Csr(G)$ were finitely generated as $\mathbb{C}$-algebras, this would be an immediate consequence of the work of Baum and Nistor \cite[Theorem~2]{BaumNistor}, as in Remark~\ref{rk-stratified-induces-bijection}. However, the algebras $\Cs(G_0)$ and $\Csr(G)$ are not finitely generated, so Baum and Nistor's ``automatic homeomorphism'' theorem  does not apply in our context. We do not know whether a similar property holds without the finite-type hypothesis in a $\Cs$-algebraic framework. 

Nevertheless, it is already {known} that the Mackey bijection \emph{does} induce a homeomorphism between the the spectrum of $I_n^0/I_{n-1}^0$ and that of $I_n/I_{n-1}$: see \cite[Corollary 3.2]{Mackey_CK_isom}. 

The corresponding statement about the spectra of $\Ups_k^0/\Ups_{k-1}^0$ and  $\Ups_k/\Ups_{k-1}$ is strictly stronger than \cite[Corollary 3.2]{Mackey_CK_isom}. It gives further insight into the relationship between the geometries of $\hatGtemp$ and $\hatGzero$, and to the best of our knowledge it was not known. We shall give a purely representation-theoretic proof here.

\begin{theorem}\label{prop-piecewise-homeo}
Fix $R\geq 0$, and let $\hatGzero(R)$ and $\hatGtemp(R)$ be the subsets of $\hatGzero$ and~$\hatGtemp$ consisting of the irreducible representations~$\pi$ such that $\Knorm{\LKT(\pi)}=R$. The Mackey bijection $\mathcal{M}\colon \hatGzero \longrightarrow\hatGtemp$ induces a homeomorphism between  $\hatGzero(R)$ and $\hatGtemp(R)$.
\end{theorem}

\begin{proof}
The main result of \cite{Mackey_continuity} is that $\mathcal{M}^{-1}\colon \hatGtemp \longrightarrow \hatGzero$ is continuous. Therefore 
it is enough to check that the restriction of $\mathcal{M}$ to $\hatGzero(R)$ is continuous. 

Given the description of the topology on  $\hatGzero$ that arises from Theorem~\ref{th-mackey-G0}(c), it is clear that every point of $\hatGzero(R)$ has a countable neighborhood basis. Therefore the continuity of~$\mathcal{M}$ on $\hatGzero(R)$ can be checked in terms of convergent sequences, see  \cite[Chapter~IX, \S\,2, n°6]{TG9}.

Let $\varrho_\infty$ be an element of $\hatGzero(R)$, and let $(\varrho_n)_{n \in \N}$ be a sequence of elements of $\hatGzero(R)$  that admits $\varrho_\infty$ as a limit point. 
What we have to check is that the sequence $(\mathcal{M}(\varrho_n))_{n \in \N}$ of irreducible tempered representations of~$G$ admits $\mathcal{M}(\varrho_\infty)$ as a limit point. 

After passing to a subsequence, and choosing realizations of the representations as in Section~\ref{sec-params-G0}, we may assume 
\begin{align*} \varrho_n = \Ind_{K_{I_\seq} \ltimes \mathfrak{p}}^{G_0}\left(\mu_{\seq}\otimes e^{i\nu_n}\right) 
\quad \text{and} \quad \varrho_\infty  = \Ind_{K_{I_\infty} \ltimes \mathfrak{p}}^{G_0}\left(\mu_{\infty}\otimes e^{i\nu_\infty}\right)
 \end{align*}
  where $I_\seq$ and $I_\infty$ are subsets of the set~$S$ of simple roots, where $\nu_\infty$ is an element of $\mathfrak{a}_{I_{\infty}, +}^\ast$, where $(\nu_n)_{n \in \N}$ is a sequence of elements of $\mathfrak{a}_{I_{\seq}, +}^\ast$, and where $\mu_\seq$ and $\mu_{\infty}$ are irreducible representations of $K_{I_\seq}$ and $K_{I_\infty}$, respectively.
  We shall view $\nu_\infty$ and the $\nu_n$s as elements of~$\adomstar$, as in Section~\ref{sec-iminfch}, and abbreviate $K_{I_\infty}$ and $K_{I_{\seq}}$ to $K_\infty$ and $K_\seq$.
  The hypothesis that $\varrho_\infty$~is a limit point of $(\varrho_n)_{n \in \N}$ can then be unpacked using Theorem~\ref{th-mackey-G0}(c) above, which summarizes  Baggett's description of the topology of $\hatGzero$ \cite{Baggett}  (see also \cite[Section~3.3]{Mackey_continuity}):
  it means, first, that $(\nu_n)_{n \in \N}$ converges to $\nu_\infty$ in $\amin^\ast$; second, that $I_{\seq}$ contains~$I_{\infty}$, so that $K_{\infty}$ contains $K_{\seq}$; and third,  that the restriction of $\mu_\infty$ to $K_\seq$ contains $\mu_\seq$.  By Frobenius reciprocity the last property means that
  \begin{equation}\label{eq-inclusions-kstructure} \mu_\infty \subset \Ind_{K_\seq}^{K_\infty}(\mu_\seq).\end{equation}
 By induction in stages this implies 
   \begin{equation} \label{eq-inclusions-kstructure-bis} \Ind_{K_\infty}^K(\mu_\infty) \subset \Ind_{K_\seq}^{K}(\mu_\seq).\end{equation}
The left-hand side of~\eqref{eq-inclusions-kstructure-bis} is the restriction of~$\varrho_\infty$ to $K$, and the right-hand side is the restriction of~$\varrho_n$. Since we  assumed that the sets of lowest $K$-types of $\varrho_\infty$ and $\varrho_n$ all have height~$R$, this implies that the lowest $K$-types of $\Ind_{K_\infty}^K(\mu_\infty)$ all occur as lowest $K$-types in $\Ind_{K_\seq}^{K}(\mu_\seq)$. 

Now let $\pi_\infty$ be the element~$\mathcal{M}(\varrho_\infty)$ of $\hatGtemp$; for all $n \in \N$, let $\pi_n$ be the element $\mathcal{M}(\varrho_n)$. Because the Mackey bijection preserves lowest $K$-types, the previous observation on the height of $K$-types implies
\begin{equation}\label{eq-inclusion-LKTs} \LKT(\pi_{\infty}) \subset \LKT(\pi_n)
\end{equation}
for all $n \in \N$. 

Let us try to simplify the notation by rewriting the parabolic subgroup $P_{I_\seq}=M_{I_\seq}A_{I_\seq} N_{I_\seq}$ of~$G$ as $P_{\seq}=M_\seq A_\seq N_\seq$, and writing $P_{\infty}= M_\infty A_\infty N_\infty$ instead of $P_{I_\infty}=M_{I_\infty} A_{I_\infty} N_{I_\infty}$. Choosing realizations for the representations as in Section~\ref{sec-parametrisation-bhy}, we may assume 
\[ \pi_n=\Ind_{P_{\seq}}^{G}\left(\tau_{\seq}\otimes e^{i\nu_n}\right)\]
for all $n \in \N$, where $\tau_\seq$ is a tempiric representation of~$M_{\seq}$ with lowest $K_{\seq}$-type $\mu_\seq$. Since $\nu_n$ converges to $\nu_\infty$ as $n \to \infty$, it follows that all irreducible factors of 
\begin{equation} \label{eq-def-tilde-pi-infty}\widetilde{\pi}_{\infty}=\Ind_{P_{\seq}}^G\left(\tau_\seq \otimes e^{i \nu_\infty}\right)\end{equation}
occur as limit points of $(\pi_n)_{n \in \N}$: see  \cite[Proposition 2.4]{Mackey_continuity}, which is based on \cite[Théorème~2.6]{Delorme86}.

\medskip

\noindent To finish the proof, it is therefore enough to establish the following assertion. 
\begin{claim*} $\pi_\infty$ is one of the irreducible factors of~$\widetilde{\pi}_\infty$. 
\end{claim*}

Let us prove this. Since $\widetilde{\pi}_\infty$ and the $\pi_n$ have the same $K$-module structure, it follows from~\eqref{eq-inclusion-LKTs} that
\begin{equation}\label{eq-inclusion-LKTs-bis}
\LKT(\pi_\infty) \subset \LKT(\widetilde{\pi}_\infty).\end{equation}
To compare $\pi_\infty$ and $\widetilde{\pi}_\infty$ in more detail, let $\tau_\infty$ be a tempiric representation of~$M_\infty$ with lowest $K_\infty$-type $\mu_\infty$; by definition of the Mackey bijection we have 
\begin{equation} \label{eq-reecriture-pi-infty} \pi_\infty \simeq \Ind_{P_\infty}^G\left(\tau_\infty \otimes e^{i\nu_\infty}\right). \end{equation} 
We can bring~\eqref{eq-def-tilde-pi-infty} and~\eqref{eq-reecriture-pi-infty} closer by rewriting $\widetilde{\pi}_\infty$ as parabolically induced from~$P_{\infty}$. Under our assumptions, the intersection $\widetilde{P}=P_\seq \cap M_{\infty}$ is a parabolic subgroup of the reductive group~$M_\infty$: as in \cite[Proof of Theorem 3.5]{Mackey_bijection} (see also  \cite[Section~4.1]{Mackey_continuity})  we may choose subgroups $\widetilde{A}$ and $\widetilde{N}$ of~$G$ such $A_{\seq} = A_\infty \widetilde{A}$, $N_\seq = N_\infty \widetilde{N}$; then $\widetilde{P}=M_\seq \widetilde{A} \widetilde{N}$ is a parabolic subgroup of~$M_\infty$, and  by parabolic induction in stages we have
\begin{equation} \label{eq-widetilde-pi}\widetilde{\pi}_{\infty}=\Ind_{P_{\infty}}^G\big(\Ind_{\widetilde{P}}^{M_\infty}(\tau_\seq \otimes 0) \otimes e^{i \nu_\infty}\big).
\end{equation}
In other words
\begin{equation} \label{eq-widetilde-pi-bis}\widetilde{\pi}_{\infty}=\Ind_{P_{\infty}}^G\big(\Sigma\otimes e^{i \nu_\infty}\big),
\end{equation}
where $\Sigma = \Ind_{\widetilde{P}}^{M_\infty}(\tau_\seq)=\Ind_{M_\seq \widetilde{A} \widetilde{N}}^{M_\infty}(\tau_\seq)$. The representation $\Sigma$ is a direct sum of tempiric representations of $M_\infty$, and its restriction to $K_\infty$ is equivalent to $\Ind_{K_\seq}^{K_\infty}((\tau_\seq)|_{K_\seq})$. By \eqref{eq-inclusions-kstructure} we see that the $K_\infty$-type $\mu_\infty$ occurs in~$\Sigma$. 

We now point out that $\mu_\infty$ is in fact one of the lowest $K$-types of $\Sigma$, and that $\tau_\infty$ is equivalent with one of the irreducible factors of~$\Sigma$. To see this, let us write $\Sigma$ as  $\Sigma_1 \oplus \cdots \oplus \Sigma_k$ where the $\Sigma_j$ are tempiric representations of $M_\infty$.
Without loss of generality we may assume that $\mu_\infty$ occurs in $\Sigma_1$. We then have 
\[ \Ind_{K_\infty}^K\left(\mu_\infty\right) \subset \Ind_{K_\infty}^K\left((\Sigma_1)|_{K_\infty}\right).\] 
Furthermore, the lowest $K$-types on the left hand side also have minimal height in the right-hand side, because the lowest $K$-types on both sides have height~$R$: indeed, the lowest $K$-types on the left-hand side are those of $\pi_\infty$ by Remark~\ref{rk-restriction-ktypes}, and those on the right-hand side are a subset of the lowest $K$-types of $\widetilde{\pi}_{\infty}$, which have the same height as those of~$\pi_\infty$ by~\eqref{eq-inclusion-LKTs-bis}. 
Now, if the unique lowest $K_\infty$-type of $\Sigma_1$ were distinct from $\mu_\infty$, then by Lemma~\ref{lem-technical-length} the lowest $K$-types of $\Ind_{K_\infty}^K(\mu_\infty)$ would not have minimal height in $\Ind_{K_\infty}^K((\Sigma_1)|_K)$. Combined with the preceding observations, this would yield a contradiction.  
 Therefore $\mu_\infty$ is a lowest $K_\infty$-type in the tempiric representations $\Sigma_1$ and $\tau_\infty$, which are therefore equivalent. Since the lowest $K$-types of~$\Sigma$ and the lowest $K$-types of all the $\Sigma_j$ have the same height, this proves our assertion that $\Sigma$ contains~$\mu_\infty$ as a lowest  $K_\infty$-type, and $\tau_\infty$ is one of the irreducible factors of~$\Sigma$.
 By~\eqref{eq-reecriture-pi-infty} and~\eqref{eq-widetilde-pi-bis},  we deduce that $\pi_\infty$ is one of the irreducible factors of $\widetilde{\pi}_\infty$, as claimed. This concludes the proof of Theorem~\ref{prop-piecewise-homeo}.
\end{proof}

The theorem has the following easy consequence, which we highlighted as Theorem~\hyperref[th-b]{B} in the introduction. 

\begin{corollary} \label{cor-homeo-composantes} The Mackey bijection maps every connected component of $\hatGtemp$ homeomorphically onto its image in $\hatGzero$.
\end{corollary}

\begin{proof}
    Each connected component of  $\hatGtemp$ is contained in $\hatGtemp(R)$ for a single  number~$R$ (see e.g. \cite{Mackey_continuity}, Section~2.3); it is therefore a locally closed subset of $\hatGtemp(R)$, and the corollary is immediate from this and from the Theorem.
\end{proof}

\begin{remark} In the notation of Theorem~\ref{prop-piecewise-homeo}, the set $\hatGtemp(R)$ is a finite union of connected components of $\hatGtemp$: this follows from the results recalled in \cite[Section~2.3]{Mackey_continuity}, and from the fact that among the possible sets $E \in \mathcal{E}$ of lowest $K$-types for an irreducible representation of~$G$, only a finite number of sets~$E$ have height $R$. Thus the contents of Theorem~\ref{prop-piecewise-homeo} and Corollary~\ref{cor-homeo-composantes} are essentially identical. \end{remark}

Let us also point out that Theorem~\ref{prop-piecewise-homeo} implies the corresponding statement for the spectra of $I_n^0/I_{n-1}^0$ and $I_n/I_{n-1}$, which was already known \cite[Corollary~3.2]{Mackey_CK_isom}.

\begin{corollary} Let~$E$ be a finite collection of~$K$-types, and let $\hatGtemp(E)$ and $\hatGzero(E)$ be the subsets of  $\hatGtemp$ and $\hatGzero$ consisting of those representations whose set of lowest $K$-types is exactly~$E$. The Mackey bijection induces a homeomorphism between $\hatGtemp(E)$ and~$\hatGzero(E)$.
\end{corollary}

\begin{proof} Set $R=\Knorm{E}$; then it is an easy consequence of Lemma~\ref{lem-Ktype-ouvert} that $\hatGtemp(E)$ and $\hatGzero(E)$ are locally closed subsets of $\hatGtemp(R)$ and $\hatGzero(R)$, respectively, and the corollary follows. \end{proof}

\section{Localizations of the Mackey embedding}\label{sec-localizations}

The definition of the Mackey embedding $\alpha\colon \Cs(G_0) \longrightarrow \Csr(G)$ in \cite{CHR} is somewhat intricate, and necessitates rather delicate properties of intertwining operators. In contrast, the recent work of Bradd, Higson and Yuncken  features apparently much simpler $\Cs$-morphisms that are analogues of the Mackey embedding, but are defined only on certain subquotients of $\Csr(G)$ and $\Cs(G_0)$---for filtrations of these two algebras that are defined in terms of the  imaginary part of the infinitesimal character. These ``local'' versions of the embedding are key ingredients for the proofs of the main results of~\cite{BHY}. 

In this section we compare the global Mackey embedding of \cite{CHR} and the local versions in \cite{BHY}, and reflect on the (non-)uniqueness of an embedding of $\Cs(G_0)$ into~$\Csr(G)$ satisfying the spectral property~\eqref{eq-spectral-CHR}.

\subsection{The infinitesimal character filtration; localized embeddings}

\subsubsection{\texorpdfstring{The stratification and subquotients for~$G$}{The stratification and subquotients for G}} \label{sec:notation_BHY}
\label{sec:general_notation}

 For $I \subset S$, define an open subset $\UI$ of $\adomstar$ as in \cite[Definition 3.2.1]{BHY}, by 
 \[ \UI = \bigcup_{J \subset I} \aiplus.\] Because the map $\Iminfch\colon \hatGtemp \longrightarrow \adomstar$ is continuous \cite[Theorem 3.1.1]{BHY}, the set $\Iminfch^{-1}(U_I)$ is open in $\hatGtemp$. 
 
 \begin{definition} Given a subset $I$ of~$S$, 
 \begin{enumerate}[(i)]
\item Let  $\Csr(G; U_I) \subset \Csr(G)$ be the annihilator of $\hatGtemp\setminus(\Iminfch^{-1}(U_I))$; that is, the ideal 
\[ \Csr(G; U_I) =\left\{\ f \in \Csr(G) \ : \ \text{$\pi(f)=0$ for all irred. tempered $\pi$ with  $\Iminfch(\pi)\notin \UI$}\ \right\}. \] 
 
\item Let $\Csr(G;I)$ be the quotient of $\Csr(G; \UI)$ by the ideal consisting of all elements that vanish in every representation $\pi$ with $\Iminfch(\pi)\in \UI\setminus \aiplus$.
\end{enumerate}
 We shall write $\proj_I\colon \Csr(G; U_I)\longrightarrow \Csr(G; I)$ for the canonical projection.
\end{definition}

\begin{remark}\label{rk-notations-filtrBHY}
    Suppose a subset $I$ of~$S$ is given. Using the discussion in  Section~\ref{sec-basics-cstar}, we may identify the unitary irreducible representations of~$\Csr(G; \UI)$ with  the irreducible tempered representations of~$G$ with $\Iminfch(\pi) \in \UI$; and the  unitary irreducible representations of~$\Csr(G; I)$ with the irreducible tempered representations of~$G$ with $\Iminfch(\pi) \in \aiplus$. 
    In particular, the spectrum of $\Csr(G; I)$ consists precisely of the equivalence classes of the representations $\pi_{I, \tau, \nu}$, as $\tau$ ranges over the tempiric representations of $M_I$ and $\nu$ ranges over~$\aiplus$. 

    Given a tempiric representation $\tau$ of~$M_I$ and an element $\nu$ of $\aiplus$, we shall use the same notation $\pi_{I, \tau, \nu}$ for the representation of~$G$ on $\Ind H_\tau$ introduced in Section~\ref{sec-parametrisation-bhy}, for the representation of $\Csr(G; \UI)$ that is obtained by restricting the corresponding representation of $\Csr(G)$ to $\Csr(G; \UI)$, and for the representation of $\Csr(G; I)$ that is obtained from the latter by passing to quotients. 
    Thus the spectrum of $\Csr(G; I)$ can be identified (as a set) with $\hatMItempiric\times \aiplus$, where $\hatMItempiric$ denotes the set of equivalence classes of tempiric representations of~$M_I$. 
\end{remark}

In fact, the following observation of Bradd, Higson and Yuncken shows that for each $I \subset S$, the subquotient $\Csr(G; I)$ has a very simple structure, and that its spectrum can be identified with $\hatMItempiric \times \aiplus$ as  a topological space, not just as a set. 

\begin{lemma}[\cite{BHY}, Theorem 3.4.2]\label{lem-bhy-quo-structure}
There is a unique $\Cs$-algebra isomorphism 
\begin{equation} \label{eq-isom-cstarG-i}\mathscr{F}_I\colon \Csr(G; I) \longrightarrow \ \bigoplus \limits_{\tau \in \widehat{M}_{I, \textrm{tempiric}}}\mathscr{C}_0\left(\aiplus, \compact\left(\Ind H_\tau\right)\right)
\end{equation}
such that, for any element $f$ of $\Csr(G;I)$, the $\tau$-component of $\mathscr{F}_I(f)$ is the function $\nu \longmapsto \pi_{I, \tau, \nu}(f)$.
\end{lemma}
In the statement, the notation $\compact(\Ind\, H_\tau)$ means the $\Cs$-algebra of compact operators on the Hilbert space~$\Ind\, H_\tau$ from~\eqref{eq-def-Htau}, the notation $\mathscr{C}_0\left(\aiplus, \compact\left(\Ind H_\tau\right)\right)$ means the $\Cs$-algebra of continuous functions $\aiplus \longrightarrow \compact\left(\Ind H_\tau\right)$ that vanish at infinity, and the $\oplus$ sign means the $\Cs$-algebraic direct sum. 

\subsubsection{\texorpdfstring{The analogues for~$G_0$}{The analogues for G0}} 

For every subset $I$ of~$S$, the ideal $\Csr(G; \UI)$ and the subquotient $\Csr(G; I)$ have counterparts in the $\Cs$-algebra $\Cs(G_0)$. The infinitesimal character map $\Iminfch\colon \hatGzero\longrightarrow \adomstar$ is, again, continuous \cite[Lemma 3.1.5]{BHY}. Thus we may consider the ideal $\Cs(G_0; \UI) \subset \Cs(G)$ defined as the annihilator of $\hatGzero \setminus \Iminfch^{-1}(\UI)$, and the quotient $\Cs(G_0; I)$ of $\Cs(G_0; \UI)$ by the ideal of all elements that vanish in every representation $\varrho$ with $\Iminfch(\varrho) \in \UI\setminus \aiplus$. Thus the spectrum of $\Cs(G_0; \UI)$ may be identified with the set of elements $\varrho \in \hatGzero$ such that $\Iminfch(\varrho) \in \UI$, and the spectrum of $\Cs(G_0; I)$ with the set of all representations $\varrho_{I, \mu, \nu}$ for $\mu \in \hatKI$ and $\nu \in \aiplus$. We shall take up the same notational conventions as in Remark~\ref{rk-notations-filtrBHY} concerning the identification of $\varrho_{I, \mu, \nu}$ as a representation of $G$, of $\Cs(G_0; \UI)$, or of $\Cs(G_0; I)$. 

Again, the subquotients $\Cs(G_0; I)$ have a very simple structure: there is a unique $\Cs$-algebra isomorphism 
\begin{equation} \label{eq-isom-cstarGzero-i} \mathscr{F}^0_I\colon \Cs(G_0; I) \longrightarrow \ \bigoplus \limits_{\mu \in \widehat{K}_{I}}\mathscr{C}_0\left(\aiplus, \compact\left(\Ind H_\mu\right)\right)
\end{equation}
such that, for any element $f$ in $\Cs(G_0; I)$, the $\mu$-component of the Fourier transform $\mathscr{F}_I^0(f)$ is the function $\nu \longmapsto \varrho_{I, \mu, \nu}$. For all this, see \cite[Section 4.3]{BHY}.

\subsubsection{Local version of the Mackey embedding} \label{sec-local-embedding}

\begin{proposition}[\cite{BHY}, Lemma 5.4.3]\label{prop-bhy-local}
Fix $I \subset S$. There exists a unique $\Cs$-morphism 
\begin{equation}\label{eq-BHY-embedding-bis}  \alpha^{I}\colon \Cs(G_0; I) \longrightarrow  \Csr(G; I)\end{equation}
with the property that 
\begin{equation} \label{eq-spectral-BHY-bis} \pi_{I,\tau, \nu}\big(\alpha^{I}(\varphi)\big) = \varrho_{I, \tau, {\nu}}(\varphi) \quad \forall \varphi \in  \Cs(G_0; I)\end{equation}
for any tempiric representation $\tau$ of~$M$ and any $\nu \in \aiplus$. 
\end{proposition}
\begin{remark} As was the case for Theorem~\ref{th-chr} the construction is explicit, but here it is much simpler because of the isomorphisms \eqref{eq-isom-cstarG-i} and~\eqref{eq-isom-cstarGzero-i}. One essential difference with Theorem~\ref{th-chr} is the uniqueness statement in the proposition, which comes from the fact that the spectral property~\eqref{eq-spectral-BHY-bis} is an equality of operators rather than a unitary equivalence of representations, and that the representations $\pi_{I, \tau, \nu}$ of $\Csr(G; I)$ separate points. This difference between Proposition~\ref{prop-bhy-local} and Theorem~\ref{th-chr} will be a key point in this section. \end{remark}

\subsection{\textbf{The compatibility between the global and local embeddings}}\label{sec-compatibility}

 Henceforth we shall assume that the Mackey embedding $\alpha\colon \Cs(G_0)\longrightarrow \Csr(G)$ is constructed as in \cite{CHR}. Therefore it comes with continuous families of intertwining operators, as in Remark~\ref{rk-chr-constructif}: for each $I \subset S$ and every tempiric representation $\tau$ of $M_I$, the construction in $\cite{CHR}$ provides a continuous family $(\mathcal{B}_{\tau,\nu})_{\nu \in \mathfrak{a}_I^\ast}$ of unitary automorphisms of $\Ind H_\tau$ such that 
 \begin{equation} \label{eq-spectral-CHR-refined} 
\pi_{I, \tau, \nu} \left( \alpha (g)\right) = \mathscr{B}_{\tau, \nu}\ \varrho_{I, \tau, \nu}(g) \ \mathscr{B}_{\tau, \nu}^{-1} \qquad \forall g \in \Cs(G_0).
 \end{equation} 

For each $I \subset S$, we may use these families to construct an automorphism of $\Csr(G; I)$, as follows. Calling in the Fourier transform $\mathscr{F}_I$ in Lemma~\ref{lem-bhy-quo-structure}, we define $\mathscr{B}_I\colon \Csr(G; I) \longrightarrow \Csr(G; I)$ to be the unique map with the following property: for every $\varphi \in \Csr(G; I)$ and  $\tau \in \hatMItempiric$, if we denote by $\widehat{\varphi}_\tau$ the $\tau$-summand of $\mathscr{F}_I(\varphi)$, then the $\tau$-summand of $\mathscr{F}_I(\mathcal{B}_{I}(\varphi))$ is the function 
\begin{align*} \aiplus & \longrightarrow \operatorname{End}\left(\Ind\,H_\tau\right) \\ \nu & \longmapsto \mathscr{B}_{\tau, \nu} \widehat{\varphi}_\tau(\nu) \mathscr{B}_{\tau, \nu}^{-1}.\end{align*} 
Because the operators $\mathcal{B}_{\tau, \nu}$ are unitary and depend continuously on $\nu$, the map $\mathscr{B}_I$ is well-defined, and is a $\Cs$-algebra automorphism. 

Our next result compares the Mackey embedding $\alpha\colon \Cs(G_0) \longrightarrow \Csr(G)$ with the collection of canonical embeddings $\alpha^I\colon \Cs(G_0; I) \longrightarrow \Csr(G; I)$, as $I$ ranges over the subsets of~$S$.
 
\begin{proposition} \label{prop-compatibility-embeddings}~
\begin{enumerate}
\item[(a)] For each  $I \subset S$, the Mackey embedding $\alpha\colon \Cs(G_0)\longrightarrow \Csr(G)$  maps $\Cs(G_0; \UI)$ into $\Csr(G; \UI)$.  

\item[(b)] For each $I \subset S$, the following  diagram commutes: 
 \begin{equation*}
\begin{tikzcd}
  \Cs(G_0; \UI)\arrow{d}{\proj_I}\arrow{r}{\alpha}
  & \Cs(G; \UI)\arrow{d}{\proj_I} \\
\Cs(G_0; I)
\arrow{r}{ \mathscr{B}_I\, \alpha^{I}}  & \Csr(G; I).
\end{tikzcd}
\end{equation*}
\end{enumerate}
\end{proposition}

 \begin{proof} 
 For~(a), let us begin with $f \in \Cs(G_0; U_I)$, and set $g=\alpha(f)$.
We need to check that for every irreducible tempered representation $\pi$ with $\Iminfch(\pi)\notin U_I$, the operator $\pi(g)$ is zero.
By Theorem~\ref{th-strat-par-ImInfChar} this is equivalent to checking that for every  pair $(J,  \sigma)$  consisting of a subset $J$ of~$S$  and a tempiric representation $\sigma$ of $M_J$, and for every linear functional $\nu$ on~$\mathfrak{a}_J$, if the extension $\nu^\sharp$ of~$\nu$ to $\mathfrak{a}_{\min}$ lies in $\adomstar \setminus U_I$, then we have $\pi_{J,\sigma, \nu}(g)=0$.

Now the spectral property in Remark~\ref{rk-chr-constructif}(c) implies:
 \begin{equation} \label{eq-interm-a} \pi_{\sigma, \nu}(g) =\pi_{J,\sigma, \nu}\big(\alpha(f)\big) = \mathscr{B}_{\sigma, \nu}^{-1} \ \varrho_{J,\sigma, {\nu}}(f)\ \mathscr{B}_{\sigma, \nu}.\end{equation}
On the right-hand side, the representation $\varrho_{J,\sigma, \nu}$ is a countable direct sum of representations $\varrho_i$ which  satisfy  $\Iminfch(\varrho_i) =\nu$. 
 Since by assumption $f \in \Cs(G_0; U_I)$, we deduce  $\varrho_{J,\sigma, \nu}(f)=0$. By~\eqref{eq-interm-a} it follows that $\pi_{J,\sigma, \nu}(g)=0$, and part~(a) is proved.

\medskip

 For part~(b), we use the following observation. \emph{Suppose $\varphi$, $\varphi'$ are elements of the bottom-right corner $\Csr(G; I)$ of the diagram. Then $\varphi=\varphi'$ if and only if for every  tempiric representation $\tau$ of~$M_I$ and for every $\nu\in \aiplus$, we have $\pi_{I,\tau, \nu}(\varphi)=\pi_{I,\tau, \nu}(\varphi')$.}

This is because the spectrum of $\Csr(G; I)$ consists precisely of the representations $\pi_{I,\tau, \nu}$ in the previous sentence, and because irreducible representations of a $C^\ast$-algebra separate points.

 \medskip

  Now let $f$ be an element of $\Cs(G_0; \UI)$. Assertion~(b) is that in the bottom-right algebra $\Csr(G; I)$, the elements $\varphi=\proj_I\left(\alpha(f)\right)$ and $\varphi'=\mathscr{B}_I\big(\alpha^{I}\left(\proj_I(f)\right)\big)$ are equal. For every tempiric representation $\tau$ of $M_I$, and for every $\nu$ in $\aiplus$, we have
\begin{align*}  \pi_{I, \tau, \nu}(\varphi')
 &=\pi_{I, \tau, \nu}\big(\mathscr{B}_I(\alpha^I(\proj_I(f))\big) &&\text{by
  definition of $\varphi'$ }\\ 
 &= \mathscr{B}_{\tau, \nu}\,\pi_{I,\tau, \nu}\left(\alpha^I(\proj_I(f)\right)\,\mathscr{B}_{\tau, \nu}^{-1} &&\text{by definition of $\mathscr{B}_I$ and by~\eqref{eq-isom-cstarG-i}} \\ 
 &= \mathscr{B}_{\tau, \nu} \,\varrho_{I, \tau, \nu}\left(\proj_I(f)\right)\mathscr{B}_{\tau, \nu}^{-1}&& \text{by \eqref{eq-spectral-BHY-bis}} \\ 
 & =  \mathscr{B}_{\tau, \nu}\, \varrho_{I, \tau, \nu}(f)\,\mathscr{B}_{\tau, \nu}^{-1}&& \text{in our notation (see Remark~\ref{rk-notations-filtrBHY})}\\ 
 &= \pi_{I, \tau, \nu}\left(\alpha(f)\right)&& \text{by \eqref{eq-spectral-CHR-refined}}\\ 
 &= \pi_{I,\tau, \nu}(\varphi) &&\text{by definition of $\varphi$}.
\end{align*}
Using the previous observation this concludes the proof. 
 \end{proof}

 \begin{remark} The Mackey embedding $\alpha\colon \Cs(G_0)\longrightarrow \Csr(G)$ in \cite{CHR} actually arises from a one-parameter family of ``rescaling automorphisms'' $\alpha_t\colon \Csr(G)\longrightarrow \Csr(G)$, indexed by $t >0$, by taking an appropriate limit as $t$ goes to zero. The rescaling automorphisms $\alpha_t$ satisfy $\pi_{I, \tau, \nu} \circ \alpha_t \simeq \pi_{I, \tau, t^{-1} \nu}$ for all $t>0$, all $I \subset S$, all $\tau \in \hatMItempiric$, and all $\nu \in \mathfrak{a}_I^\ast$. At the level of the subquotients, there is also a local version of these rescaling automorphisms: there is a \emph{unique} one-parameter family $(\alpha_t^I)_{t>0}$ of automorphisms of $\Csr(G, I)$ such that $\pi_{I, \tau, \nu} \circ \alpha_t^I = \pi_{I, \tau, t^{-1} \nu}$ for all $t>0$, all $\tau \in \hatMItempiric$, and all $\nu \in \aiplus$. Here again, the spectral property of the local version is an equality of operators, which characterizes it uniquely, whereas the global version takes the form of a unitary equivalence. Using a family of intertwining operators $\big(\mathscr{B}_{\tau, \nu, t}\big)_{\nu \in \mathfrak{a}_I^\ast}$ for fixed $t$, as in \cite[Section~6]{CHR}, one may also formulate a compatibility statement between $\alpha_t$ and the local versions $\alpha_t^I$. This is slightly more precise than Proposition~\ref{prop-compatibility-embeddings}, since the latter can then be deduced  from this compatibility statement by taking the limit as $t \to 0$. We omit the details. 
 \end{remark}

\subsection{A converse statement and further remarks}\label{sec-converse-and-remarks}

The following simple lemma  shows that the compatibility properties~(a) and~(b) in Proposition~\ref{prop-compatibility-embeddings}, concerning the behavior of the Mackey embedding $\alpha$ on the ideals and subquotients of the infinitesimal character filtration, are enough to characterize the embedding $\alpha$. 

\begin{lemma}\label{lem-localglobal-bhy-bis} Suppose $\beta\colon \Cs(G_0) \longrightarrow \Csr(G)$ is a $\Cs$-morphism that satisfies:
\begin{enumerate}
    \item[(i)] For every subset $I$ of~$S$, the morphism $\beta $ maps $\Cs(G_0; \UI)$ into $\Csr(G; \UI)$; 
    \item[(ii)] For every subset $I$ of~$S$, the following diagram is commutative: 
    \begin{equation*}
\begin{tikzcd}
  \Cs(G_0; U_I)\arrow{d}{\proj_I}\arrow{r}{\beta}
  & \Csr(G; U_I)\arrow{d}{\proj_I} \\
\Cs(G_0; I)
\arrow{r}{\mathscr{B}_I\, \alpha^{I}} & \Csr(G; I).
\end{tikzcd}
\end{equation*}
    \end{enumerate}
    Then $\beta=\alpha$.
\end{lemma}

\begin{proof} 
Because the irreducible representations of~$\Csr(G)$ separate points, it is enough to check that for every irreducible tempered representation $\pi$ of $G$, the representations $\pi \circ \alpha$ and $\pi \circ \beta$ of $\Cs(G_0)$ are identical. 

Let $\pi$ be an irreducible tempered representation of~$G$. Up to unitary equivalence we may assume $\pi = \pi_{I, \tau, \nu}$ where $I \subset S$, where $\tau$ is a tempiric representation of~$M_I$, and $\nu \in \aiplus$. 

Suppose $f$ is an element of $\Cs(G_0; \UI)$ and $\varphi = \proj_I(f)$ is its image in $\Cs(G_0; I)$. We have 
\begin{align*}
\pi(\beta(f))& =\pi\big(\proj_I\left(\beta(f)\right)\big) && \text{in view of our notational conventions} \\
&= \pi\big(\mathscr{B}_I \alpha^I(\varphi)\big) &&  \text{ by Property~(ii) in the diagram}\\
&= \pi\left(\alpha(f)\right) &&  \text{ by Proposition~\ref{prop-compatibility-embeddings}(b).}
\end{align*}

Therefore $\pi \circ \alpha$ and $\pi\circ \beta$ are two nondegenerate representations of $\Cs(G_0)$ which coincide on the ideal $\Cs(G_0; \UI)$. By \cite[Proposition~2.10.4]{Dixmier} they must be equal, q.e.d.
\end{proof}

\begin{remark}\label{rk-discussion-unicite}
   Combining Lemma~\ref{lem-localglobal-bhy-bis} and Proposition~\ref{prop-compatibility-embeddings}, we obtain the following picture for the relationship between the Mackey embedding $\alpha\colon \Cs(G_0)\longrightarrow \Csr(G)$, whose construction depends on choices  in \cite{CHR},  and the local versions $\alpha^{I}\colon \Cs(G_0; I)\longrightarrow \Csr(G; I)$, which are canonical but defined only at the level of subquotients of the infinitesimal character filtration. 
   \begin{itemize}
       \item The embeddings $\alpha_I$ can be combined together to form the (``global'') Mackey embedding~$\alpha$, using the automorphisms $\mathscr{B}_I$, which therefore encapsulate all the choices made in \cite{CHR}.
       \item Once the automorphisms $\mathscr{B}_I$ are given, there is only one way to combine the canonical embeddings $\alpha^I$ into an embedding $\alpha\colon \Cs(G_0)\longrightarrow \Csr(G)$ that satisfies the spectral property~\eqref{eq-spectral-CHR}.
   \end{itemize}
\end{remark}

\begin{remark}\label{rk-non-unicite-du-plongement} When~$G$ is a complex semisimple Lie group, the earlier work of Higson and Rom\'an \cite{HigsonRoman} shows that each of the automorphisms  $\mathscr{B}_I$ can  be chosen to be the identity of~$\Csr(G; I)$. In general it is not possible to make such a simple choice, for the following reason. In \cite[Section~6]{CHR}, the operators $\mathscr{B}_{\tau, \nu}\colon \Ind H_\tau \longrightarrow \Ind H_\tau$ are constructed to be Knapp--Stein intertwiners between certain parabolically induced representations of~$G$. Unless all the relevant parabolically induced representations are irreducible (which is the case for complex groups), it is not possible for these operators $\mathscr{B}_{\tau, \nu}$ to be all scalar, and therefore the automorphisms $\mathscr{B}_I$ cannot all be trivial. As a result,  Remark~\ref{rk-discussion-unicite} is the most we can currently say about the (non-)uniqueness of the embedding $\alpha$, and its dependence on the choices in \cite{CHR}.     \end{remark}

\bibliographystyle{amsalpha}
\bibliography{bibliov2}

\end{document}